\documentclass{amsart}
\usepackage{amsmath,amssymb,amsthm,mathtools}
\usepackage{enumitem, tikz}
\usepackage{hyperref}

\usepackage{nicematrix}
\newtheorem{theorem}{Theorem}[section]
\newtheorem{lemma}[theorem]{Lemma}
\newtheorem{proposition}[theorem]{Proposition}
\newtheorem{corollary}[theorem]{Corollary}
\numberwithin{equation}{section}

\theoremstyle{definition}
\newtheorem{definition}[theorem]{Definition}
\newtheorem{example}[theorem]{Example}

\theoremstyle{remark}
\newtheorem{remark}[theorem]{Remark}

\newcommand{\Z}{\mathbb Z}
\newcommand{\R}{\mathbb R}
\newcommand{\Q}{\mathbb Q}
\newcommand{\cF}{\mathcal F}
\newcommand{\RZ}{\mathbb R\mathcal Z}

\newcommand{\one}{\mathbf 1}
\newcommand{\field}{\Bbbk}
\def\genus{\mathfrak g}

\DeclareMathOperator{\join}{\ast}
\DeclareMathOperator{\supp}{supp}
\DeclareMathOperator{\row}{row}
\DeclareMathOperator{\im}{im}
\DeclareMathOperator{\id}{id}
\DeclareMathOperator{\link}{lk}
\DeclareMathOperator{\rank}{rank}

\title{Symplectic small covers in dimension four}

\author[S. Choi]{Suyoung Choi}
\address{Department of mathematics, Ajou University, 206, World cup-ro, Yeongtong-gu, Suwon 16499, Republic of Korea}
\email{schoi@ajou.ac.kr}

\date{\today}
\thanks{This work was supported by the National Research Foundation of Korea Grant funded by the Korean Government (RS-2025-00521982).}



\begin{document}

\begin{abstract}
    We study symplectic structures on four-dimensional small covers.
    Our main result shows that every symplectic four-dimensional small cover is aspherical.
    We then classify symplectic small covers over products of two polygons, proving that symplecticity is equivalent to factor-compatibility.
    We also classify them up to diffeomorphism.
    Finally, we construct a symplectic four-dimensional small cover whose orbit polytope is not combinatorially equivalent to a product of two polygons.
\end{abstract}

\maketitle


\section{Introduction}

\emph{Small covers}, introduced by Davis and Januszkiewicz~\cite{Davis-Januszkiewicz1991}, are real analogues of quasitoric manifolds.
They are closed manifolds equipped with locally standard $\Z_2^n$-actions whose orbit spaces are simple convex polytopes.
They form one of the basic classes of manifolds in toric topology, and their topology is controlled by a simple polytope together with a mod~$2$ characteristic map.
For general background on toric topology, we refer the reader to~\cite{Buchstaber-Panov2015book}.

From the viewpoint of symplectic geometry, small covers are often regarded as real or Lagrangian objects associated with toric geometry.
This point of view is natural, but it leaves open the question of when a small cover itself carries a symplectic structure.
A fundamental positive result in this direction is due to Ishida~\cite{Ishida2011}.
For \emph{real Bott manifolds}, equivalently small covers over cubes, he proved that cohomological symplecticity is equivalent to the existence of a symplectic structure.

The goal of the present paper is to study symplectic structures on four-dimensional small covers.
In dimension four, c-symplecticity is easy to detect.
For a closed connected smooth $4$-manifold, it is equivalent to orientability together with $b_2>0$.
Thus c-symplecticity is flexible, whereas symplecticity is much more rigid. 

The symplecticity problem for four-dimensional small covers is related to the geography of aspherical and hyperbolic four-manifolds.
For instance, small covers of the right-angled $120$-cell give many closed hyperbolic four-manifolds, but it is unknown whether any of them is symplectic.
LeBrun conjectured that closed hyperbolic four-manifolds have vanishing Seiberg--Witten invariants~\cite{LeBrun2002}; together with Taubes' non-vanishing theorem~\cite{Taubes1994}, this would rule out symplectic structures on such manifolds.
This viewpoint also motivates recent work on hyperbolic four-manifolds with vanishing Seiberg--Witten invariants~\cite{Agol-Lin2020,Stover2026}.

The main results of the paper may be summarized as follows.
Our first main result is a flagness obstruction.
We prove that if a four-dimensional small cover is symplectic, then the dual simplicial sphere of its orbit polytope must be flag.
Consequently, every symplectic four-dimensional small cover is aspherical.
See Corollary~\ref{cor:only_flag_support_symplectic_small}.

We also obtain a corresponding statement for real moment-angle manifolds.
If a four dimensional real moment-angle manifold $\RZ_K$ is symplectic, then either $K$ is flag or $K \cong \partial(P_3 \times P_{m-3})^\ast$.
Equivalently, apart from the exceptional case $\RZ_K \cong S^2 \times \Sigma$, where $\Sigma$ is a closed orientable surface, symplecticity forces $K$ to be flag.
See Theorem~\ref{thm:only_flag_support_symplectic_RZ}.

We then give a sharp classification for products of two polygons.
A small cover over a product of two polygons is symplectic if and only if it is factor-compatible.
See Theorem~\ref{thm-two-polygons-factor-compatible}.
We further enumerate the symplectic small covers over products of two polygons, and show that their diffeomorphism type is determiend by their mod~$2$ cohomology rings.
See Theorem~\ref{thm:two-stage-mod2-rigidity} and Corollary~\ref{cor:number_of_diffeo_types}.

Finally, we show that symplectic four-dimensional small covers are not exhausted by the examples over products of two polygons.
We construct a symplectic four-dimensional small cover whose orbit polytope is not combinatorially equivalent to a product of two polygons.
See Theorem~\ref{thm:q10-product-compatible-example}.

We also include a computer-assisted verification showing that, among small covers whose orbit polytopes have at most ten facets, every symplectic example arises from a product of two polygons and is factor-compatible.

\section{Preliminaries}\label{sec-preliminaries}

In this section we collect the basic definitions and facts that will be used throughout the paper.
Throughout the paper, we write $\Z_2=\Z/2\Z=\{0,1\}$.
For a field~$\field$ and a topological space~$X$, we write $b_i(X;\field)=\dim_\field H^i(X; \field)$.
When $\field=\Q$, we simply write $b_i(X)$.
All vector spaces over $\Z_2$ are written additively.

If $K$ is a simplicial complex on the vertex set $[m]=\{1,\dots,m\}$ and $\omega\subset [m]$, then $K_\omega$ denotes the full subcomplex of $K$
on the vertex set $\omega$.
Throughout the paper, we identify a subset of $[m]$ with its indicator vector in $\Z_2^m$, and also identify a vector $u\in\Z_2^m$ with its support
$$
    \supp(u)=\{i\in [m]\mid u_i=1\}.
$$
Thus $K_u$ means $K_{\supp(u)}$ for $u\in\Z_2^m$.
We also write $u^c$ for the complementary subset $[m]\setminus\supp(u)$, or equivalently for the vector $\one_m+u$.

\subsection{Small covers and characteristic maps} \label{subsec:smallcover}
Let $P$ be a simple polytope of dimension $n$, and let $\cF(P)=\{F_1,\dots,F_m\}$ be the set of its facets.
A mod~$2$ \emph{characteristic map} on $P$ is a map
$$
    \lambda \colon \cF(P)\to \Z_2^n
$$
such that whenever $F_{i_1}\cap \cdots\cap F_{i_k}\neq \varnothing$, the vectors $\lambda(F_{i_1}),\dots,\lambda(F_{i_k})$ are linearly independent over $\Z_2$.
We use the same symbol $\lambda$ for the associated matrix
$$
    \lambda= \begin{pmatrix} \lambda(F_1) & \cdots & \lambda(F_m) \end{pmatrix} \in M_{n\times m}(\Z_2).
$$
We also write $\row(\lambda)\subset \Z_2^m$ for the row space of $\lambda$.
With this convention, the defining condition means that the columns corresponding to the facets meeting at each vertex form a basis of $\Z_2^n$.

The \emph{small cover} $M(P,\lambda)$ over $P$ determined by $\lambda$ is defined by
$$
    M(P,\lambda):=(\Z_2^n\times P)/\sim,
$$
where $(g,p)\sim (h,q)$ if $p=q$ and $g+h\in G_{F(p)}$.
Here $F(p)$ denotes the unique face of $P$ whose relative interior contains $p$, and
$$
    G_{F}:=\langle \lambda(F_i)\mid F\subset F_i\rangle \subset \Z_2^n
$$
is the subgroup generated by the characteristic vectors of the facets containing $F$.
See \cite{Davis-Januszkiewicz1991} for details.

The group $\Z_2^n$ acts naturally on $M(P,\lambda)$ by $a\cdot [g,p]=[a+g,p]$, and the orbit space of this action is naturally identified with $P$.
Two small covers over $P$ are said to be \emph{Davis--Januszkiewicz equivalent} (or simply \emph{D--J equivalent}) if there is a homeomorphism between them
which is weakly $\Z_2^n$-equivariant with respect to an automorphism of $\Z_2^n$ and which covers the identity on $P$.
Throughout the paper, unless otherwise specified, small covers are considered up to D--J equivalence.

Elementary row operations on $\lambda$ correspond to changing the basis of $\Z_2^n$, and hence do not change the small cover up to D--J equivalence.
After relabeling the facets so that $F_1,\dots,F_n$ meet at a chosen vertex and then applying elementary row operations over $\Z_2$, we may assume that
$$
    \lambda=(I_n\mid A).
$$

We also recall the Davis--Januszkiewicz description of the mod $2$ cohomology ring.
Let $v_i\in H^1(M;\Z_2)$ be the Poincar\'e dual of the characteristic submanifold over the facet $F_i$.
Then
\begin{equation} \label{eq:mod2cohom_smallcover}
    H^\ast(M;\Z_2)\cong \Z_2[v_1,\dots,v_m]/(I_P+J_\lambda),  
\end{equation}
where $\deg v_i=1$ for all $i$.
Here, $I_P$ denotes the Stanley--Reisner ideal of $P$, that is, the ideal generated by all square-free monomials $v_{i_1}\cdots v_{i_r}$ for which the corresponding facets $F_{i_1},\dots,F_{i_r}$ have empty intersection, and $J_\lambda$ is the ideal generated by the linear forms
$$
    \sum_{i=1}^m \lambda_{ji}v_i \qquad (j=1,\dots,n),
$$
where $\lambda_{ji}$ denotes the $(j,i)$-entry of the characteristic matrix $\lambda$.

Let $f_i$ denote the number of codimension-$(i+1)$ faces of $P$.
Thus $f_{-1}=1$ corresponds to the unique codimension-$0$ face, namely $P$ itself.
Its $h$-vector $h(P)=(h_0,\ldots,h_n)$ is defined by
$$
    \sum_{i=0}^n h_i t^{n-i} = \sum_{i=0}^n f_{i-1}(t-1)^{n-i}.
$$
The $i$th mod~$2$ Betti number $b_i(M;\Z_2)$ is completely determined by face numbers of $P$, that is, 
\begin{equation} \label{eq:mod2betti}
    b_i(M;\Z_2)=h_i(P),  
\end{equation}
where $h(P)=(h_0,\dots,h_n)$ is the $h$-vector of $P$; see \cite{Davis-Januszkiewicz1991}.

\subsection{Real moment-angle complexes and real toric spaces} \label{subsection:RZK}
Let $K$ be a simplicial complex on the vertex set $[m]$.
The \emph{real moment-angle complex} $\RZ_K$ is defined by
$$
    \RZ_K = \bigcup_{\sigma\in K}(D^1,S^0)^\sigma \subset (D^1)^m,
$$
where
$$
    (D^1,S^0)^\sigma = \{(x_1,\dots,x_m)\in (D^1)^m \mid x_i\in S^0 \text{ for } i\notin \sigma\}.
$$

When $K$ is a piecewise linear~(PL)~$(n-1)$-sphere, $\RZ_K$ is a simplicial~$n$-manifold; see~\cite[Theorem~2.3]{Cai2017}.
In particular, if $n \leq 7$, $\RZ_K$ admits a smooth structure.
Throughout this paper, all simplicial spheres are finite.
In dimension three, every simplicial sphere is PL.

\begin{example} \label{ex:RZ_polygon}
    Let $K=\partial P_m^\ast$ be the boundary complex dual to an $m$-gon.
        
    We use the cubewise cell structure on $(D^1)^m=[-1,1]^m$.
    A cell of $\RZ_{K}$ is determined by a simplex $\sigma\in K$ together with a choice of signs on the complementary coordinates.
    Since $K$ has one empty simplex, $m$ vertices, and $m$ edges, the induced cell decomposition on $\RZ_{K}$ has $f_0=2^m$, $f_1=m2^{m-1}$, $f_2=m2^{m-2}$ cells in dimensions $0,1,2$, respectively.

    Since $K \cong S^1$, $\RZ_K$ is a closed surface.
    Since $\RZ_K$ is a smooth complete intersection of real quadrics in $\R^m$, its normal bundle is canonically trivialized by the gradients of the defining equations, and hence, $\RZ_K$ is stably parallelizable, in particular orientable.
     
    Let $\genus(m)$ denote the genus of $\RZ_K$.
    Since the Euler characteristic is $2^{m-2}(4 - m)$, we obtain $\genus(m) = 1 + (m-4) 2^{m-3}$.
    In particular, $m=3$ gives $S^2$, $m=4$ gives $T^2$, and $m\ge 5$ gives a higher-genus surface.

    If $K=\partial (P_{m_1} \times P_{m_2})^\ast = \partial P_{m_1}^\ast \join \partial P_{m_2}^\ast$, then
    $$
        \RZ_K \cong \RZ_{\partial P_{m_1}^\ast}\times \RZ_{ \partial P_{m_2}^\ast},
    $$
    thus $\RZ_K$ is a product of two closed orientable surfaces.
\end{example}

We say that a linear map $\lambda \colon \Z_2^m\to \Z_2^\ell$ is \emph{nonsingular} over $K$ if, for every simplex~$\sigma\in K$, the vectors $\lambda(e_i)$ for $i\in \sigma$ are linearly independent over $\Z_2$.
Such a map $\lambda$ is also called a \emph{mod~$2$ characteristic} map over~$K$.
If $P$ is a simple $n$-polytope and $K=\partial P^\ast$, then the vertices of $K$ are naturally identified with the facets of~$P$.
With this identification and with $\ell=n$, a mod-$2$ characteristic map over $K$ is precisely a mod-$2$ characteristic map over $P$ in the sense of Subsection~\ref{subsec:smallcover}.

If $\lambda$ is a mod~$2$ characteristic map over~$K$, then $\ker\lambda$ acts freely on $\RZ_K$ by~\cite{Choi-Kaji-Theriault2017}.
The quotient
$$
    M^\R(K,\lambda):=\RZ_K/\ker\lambda
$$
is called a \emph{real toric space}.

When $\lambda = \id_{\Z_2^m}$, one has $M^\R(K,\lambda)=\RZ_K$.
When $K=\partial P^\ast$ is the dual complex of a simple polytope $P$ and $\lambda\colon \Z_2^m\to \Z_2^n$ is a mod~$2$ characteristic map on $P$, this construction recovers the small cover $M(P,\lambda)$.

Note that not every simplicial~sphere is the dual complex of a simple polytope.
Thus the above construction contains all small covers in the classical sense.
However, the mod $2$ cohomology presentation~\eqref{eq:mod2cohom_smallcover} and the mod~$2$ Betti number formula~\eqref{eq:mod2betti} extend to the generalized Davis--Januszkiewicz setting; see \cite[Theorem~5.12]{Davis-Januszkiewicz1991}.
Therefore, when $K$ is a simplicial $(n-1)$-sphere and $\lambda\colon \Z_2^m\to \Z_2^n$ is nonsingular over $K$, we also call $M^\R(K,\lambda)$ a \emph{small cover over $K$}.
This agrees with the classical small cover when $K=\partial P^\ast$ for a simple polytope $P$.

One of the main cohomological tools in this paper is the cohomology formula for real toric spaces.
\begin{theorem}[{\cite{Choi-Park2020}}]\label{thm:CP}
    Let $\field$ be a commutative ring in which $2$ is a unit.
    There is a $(\Z\oplus \row(\lambda))$-graded $\field$-algebra isomorphism
    $$
        H^\ast\bigl(M^\R(K,\lambda);\field \bigr) \cong \bigoplus_{\omega\in \row(\lambda)} \widetilde H^{\ast-1}(K_\omega;\field).
    $$ 
    More precisely, the product structure on $\bigoplus_{\omega\in \row(\lambda)}\widetilde H^\ast(K_\omega;\field)$ is given by the canonical maps
    $$
        \widetilde H^{i-1}(K_{\omega_1};\field)\otimes \widetilde H^{j-1}(K_{\omega_2};\field) \to \widetilde H^{i+j-1}(K_{\omega_1+\omega_2};\field),
    $$
    which are induced by the simplicial maps
    $$
        K_{\omega_1+\omega_2}\to K_{\omega_1} \join K_{\omega_2},
    $$
    where $\join$ denotes the simplicial join.
\end{theorem}
In the special case $\lambda=\id_{\Z_2^m}$, this gives the corresponding formula for $\RZ_K$. 
For this case, Cai~\cite{Cai2017} proved the formula with integral coefficients.

In particular, for a real toric space $M=M^\R(K,\lambda)$, we have
\begin{equation} \label{eq:Hochster_decom}
    H^i(M;\Q) \cong \bigoplus_{\omega\in \row(\lambda)} \widetilde H^{i-1}(K_\omega;\Q),  
\end{equation}
and the multiplicative structure is given by Theorem~\ref{thm:CP}.
We will refer to \eqref{eq:Hochster_decom} as the \emph{Hochster-type decomposition}.

We now recall the orientability criterion for real toric spaces over simplicial~spheres.
Let $\one_m=(1,\dots,1)\in \Z_2^m$.
If $M=M(P,\lambda)$ is a small cover, then $K=\partial P^\ast$ is an $(n-1)$-sphere, so $\widetilde H^{n-1}(K;\Q)\ne 0$.
By Theorem~\ref{thm:CP}, this top class contributes to $H^n(M;\Q)$ exactly when $\one_m\in \row(\lambda)$.
Since $\one_m$ is the only weight that can contribute in top degree, one recovers the following criterion, originally due to~\cite{Nakayama-Nishimura2005} in the classical case.

\begin{corollary}\label{cor:NN}
    Let $K$ be a simplicial~sphere on $[m]$, and let $M = M^\R(K,\lambda)$ be a real toric space.
    Then $M$ is orientable if and only if the all-one vector
    $$
        \one_m=(1,\dots,1) \in \Z_2^m
    $$
    belongs to $\row(\lambda)$.
    In particular, $\RZ_K$ is orientable.
\end{corollary}

The following is an elementary linear-algebra observation.
We state it separately because it will be used several times in the sequel.

\begin{lemma}\label{lem:row-kernel-duality}
    Let $\lambda\colon \Z_2^m\to \Z_2^\ell$ be a linear map, and let $G_\lambda=\ker\lambda$.
    For $u\in \Z_2^m$, $u \not\in \row(\lambda)$ if and only if there exists $g \in G_\lambda$ such that 
    $$
        \langle u, g \rangle=1,
    $$
    where $\langle-,-\rangle$ denotes the standard pairing on $\Z_2^m$.
\end{lemma}
\begin{proof}
    This follows from the elementary identity
    $$
        \row(\lambda)=(\ker\lambda)^\perp
    $$
    with respect to the standard pairing on $\Z_2^m$.
\end{proof}

\subsection{Symplectic and basic complex-geometric terminology}
Let $N$ be a closed smooth manifold of real dimension $2n$.
A \emph{symplectic structure} on $N$ is a closed nondegenerate differential $2$-form $\Omega$ on $N$, and a manifold equipped with such a form is called
\emph{symplectic}.
If $(N,\Omega)$ is symplectic, then
$$
    [\Omega]^n\neq 0\in H^{2n}(N;\R),
$$
since $\Omega^n$ is a nowhere vanishing top-degree form.
Accordingly, $N$ is called \emph{cohomologically symplectic}, or \emph{c-symplectic}, if there exists a class $a\in H^2(N;\Q)$ such that
$$
    a^n\neq 0\in H^{2n}(N;\Q).
$$
Since rational classes are dense in $H^2(N;\R)$ and the condition $x^n\neq 0$ is open, every symplectic manifold is c-symplectic.
The converse is false in general. 

We begin with the elementary criterion for c-symplecticity in dimension four.

\begin{proposition}\label{prop-four-manifold-csymp}
    Let $N$ be a closed connected smooth $4$-manifold.
    Then $N$ is c-symplectic if and only if $N$ is orientable and $b_2(N)>0$.
    
    In particular, for a $4$-dimensional small cover~$M=M^\R(K,\lambda)$ over a simplicial~$3$-sphere~$K$, $M$ is c-symplectic if and only if  $\one_m \in \row(\lambda)$ and there exists $\omega\in\row(\lambda)$ such that $\widetilde H^1(K_\omega;\Q)\ne 0$.
\end{proposition}
\begin{proof}
    The ``only if'' part is clear.
    Indeed, a c-symplectic class $a$ and $a^2$ are nonzero in $H^2(N;\Q)$ and $H^4(N;\Q)$, respectively.
    Conversely, assume that $N$ is orientable and $b_2(N)>0$.
    By Poincar\'e duality, the intersection pairing
    $$
        Q(x,y)=\langle xy,[N]\rangle
    $$
    on $H^2(N;\Q)$ is nondegenerate.
    If every $x\in H^2(N;\Q)$ satisfied $x^2=0$, then the polarization identity would give
    $$
        Q(x,y) = \frac12\left\langle (x+y)^2-x^2-y^2,[N]\right\rangle=0
    $$
    for all $x,y$, contradicting nondegeneracy.
    Thus there exists $x\in H^2(N;\Q)$ such that $x^2\ne 0$.
    The final assertion follows by combining the first statement with Theorem~\ref{thm:CP} and Corollary~\ref{cor:NN}.
\end{proof}

\section{Basic obstructions} \label{sec-basic_obstructions}
Throughout this section, $K$ is a simplicial~$3$-sphere on the vertex set $[m]$, and $\lambda \colon \Z_2^m \to \Z_2^4$ a mod~$2$ characteristic map over $K$.
Let $M = M^\R(K,\lambda) = \RZ_K/{G_\lambda}$ be the corresponding small cover, where $G_\lambda=\ker\lambda$.

The purpose of this section is to collect several basic obstructions to the existence of symplectic structures on $M$ or $\RZ_K$.
Some of them depend only on $K$, while others depend on the canonical cover $\RZ_K$ or on the action of $G_\lambda$.

The first obstruction is elementary but fundamental.
Since every small cover is a finite quotient of its real moment-angle cover, symplecticity of the small cover forces symplecticity of the canonical cover.

\begin{proposition} \label{prop:canonical-cover-obstruction}
    Let $M = \RZ_K/G_\lambda$ be a $4$-dimensional small cover.
    If $M$ admits a symplectic structure, then $\RZ_K$ admits a $G_\lambda$-invariant symplectic structure.
    Conversely, if $\RZ_K$ admits a $G_\lambda$-invariant symplectic structure, then $M$ is symplectic.
    
    In particular, if $\RZ_K$ is not symplectic, then no small cover over $K$ is symplectic.
\end{proposition}
\begin{proof}
    Let $\pi \colon \RZ_K \to M$ be the finite covering map. 
    If $\Omega$ is a symplectic form on $M$, then $\widetilde\Omega:=\pi^\ast \Omega$ is a closed nondegenerate $2$-form on $\RZ_K$. 
    Moreover, $\widetilde\Omega$ is $G_\lambda$-invariant because $\pi$ is the quotient map by $G_\lambda$.
    
    Conversely, suppose that $\RZ_K$ carries a $G_\lambda$-invariant symplectic form $\widetilde\Omega$. 
    Since the action of $G_\lambda$ is free and finite, $G_\lambda$-invariant differential forms on $\RZ_K$ are naturally identified with differential forms on the quotient $\RZ_K/G_\lambda=M$. 
    Hence $\widetilde\Omega$ descends to a closed nondegenerate $2$-form on $M$. 
    Thus $M$ is symplectic.
\end{proof}

We shall also use the following consequence.

\begin{lemma} \label{lem:canonical-class-invariant}
    Assume that $M=\RZ_K/G_\lambda$ is symplectic, and let $\widetilde\Omega=\pi^\ast\Omega$ be the pulled-back symplectic form on $\RZ_K$. 
    Then every element of $G_\lambda$ preserves the symplectic canonical class
    $$
        K_{\widetilde\Omega}\in H^2(\RZ_K;\Z).
    $$
\end{lemma}
\begin{proof}
    Every $g\in G_\lambda$ is a symplectomorphism of $(\RZ_K,\widetilde\Omega)$. 
    Choose an almost complex structure $J$ compatible with $\widetilde\Omega$. 
    Then $g^\ast J$ is also compatible with $\widetilde\Omega$. 
    Since the space of compatible almost complex structures is contractible, the class~$K_{\widetilde\Omega}:=-c_1(T\RZ_K,J)$ is independent of the choice of~$J$. 
    Therefore $g^\ast K_{\widetilde\Omega}=K_{\widetilde\Omega}$ for every $g\in G_\lambda$.
\end{proof}

\begin{remark} \label{rem:cover-vs-quotient}
    Proposition~\ref{prop:canonical-cover-obstruction} separates two different questions. 
    First, one may ask whether the canonical cover $\RZ_K$ is symplectic at all. 
    Second, even if $\RZ_K$ is symplectic, one must ask whether it carries a $G_\lambda$-invariant symplectic form.
    This distinction will be important later.
\end{remark}

The next obstruction depends only on the face numbers of $K$.
\begin{proposition}\label{prop-wu-obstruction}
    If a $4$-dimensional small cover $M=M^\R(K,\lambda)$ admits a symplectic structure, then
    $$
        \chi(M)\equiv 0\pmod 4.
    $$
    Equivalently, the number $f_3(K)$ of facets of $K$ is divisible by $4$.
\end{proposition}
\begin{proof}
    Let $\Omega$ be a symplectic form on $M$, and choose an $\Omega$-compatible almost complex structure $J$.
    Then the first Chern class $c_1(TM,J)$ is a characteristic element for the intersection form, and the Hirzebruch--Hopf formula~\cite{Hirzebruch-Hopf1958} gives
    \begin{equation}\label{eq:almost_complex_identity}
        c_1(TM,J)^2=2\chi(M)+3\sigma(M),      
    \end{equation}
    where $\sigma(M)$ is the signature of $M$.
    
    Since $M$ is symplectic, it is orientable. 
    Hence $\one_m\in\row(\lambda)$ by Corollary~\ref{cor:NN}.
    By the Hochster-type product structure in Theorem~\ref{thm:CP}, the intersection pairing on $H^2(M;\Q)$ pairs the summand of weight $\omega$ only with the summand of weight $\omega^c=\one_m+\omega$.
    Since no weight satisfies $\omega=\omega^c$, the intersection form is a direct sum of hyperbolic pairs. 
    Thus $\sigma(M)=0$.
    Thus
    $$
        c_1(TM,J)^2=2\chi(M).
    $$
    
    By van der Blij's congruence~\cite{Blij1959}, the square of a characteristic element is congruent to the signature modulo $8$.
    Therefore
    $$
        2\chi(M) = c_1(TM,J)^2 \equiv 0\pmod 8,
    $$
    which proves the first statement.
    
    For a small cover over a simplicial $3$-sphere, by~\eqref{eq:mod2betti} we have
    $$
        \chi(M)=h_0-h_1+h_2-h_3+h_4=f_1 - 5f_0 + 16.
    $$
    For a simplicial~$3$-sphere, we have $f_1-f_0 = f_3$.
    These prove the equivalent form.
\end{proof}

Let $K$ be a simplicial~$3$-sphere on the vertex set~$[m]$.
If $K_\omega\cong S^1$, then, after fixing signs on the complementary coordinates, we obtain an embedded coordinate surface $\RZ_{K_\omega} \subset \RZ_K$.

\begin{lemma}\label{lem:coordinate-surface-dual}
    Let $K$ be a simplicial~$3$-sphere on~$[m]$, and let $\omega\subset [m]$ be such that $K_\omega\cong S^1$.
    Then the coordinate surface $\RZ_{K_\omega}\subset \RZ_K$ represents a nonzero rational homology class. 
    Moreover, its rational self-intersection is zero.
    
    In particular, for $q \geq 3$, if $K_\omega\cong \partial P_q^\ast$, then $\RZ_{K_\omega}\cong \Sigma_{\genus(q)}$ is an embedded closed orientable surface of square zero in $\RZ_K$.
\end{lemma}
\begin{proof}
    The rational cohomology of $\RZ_K$ is given by the Hochster-type decomposition~\eqref{eq:Hochster_decom}.
    Under Poincar\'e duality for $\RZ_K$, the fundamental class of the coordinate submanifold $\RZ_{K_\omega}$ has rational Poincar\'e dual in the summand indexed by the complementary subset~$\omega^c$.
    
    Since $K_\omega\cong S^1$, combinatorial Alexander duality for the simplicial~$3$-sphere $K$ gives $\widetilde H^1(K_{\omega^c};\Q)\cong \Q$.
    Thus the Poincar\'e dual of $\RZ_{K_\omega}$ is nonzero, and hence $\RZ_{K_\omega}$ represents a nonzero rational homology class.
    
    The square of this Poincar\'e dual has Hochster weight $\omega^c+\omega^c=0$.
    The weight-zero summand in degree four is $\widetilde H^3(K_\varnothing;\Q)=0$.
    Therefore the square of the Poincar\'e dual is zero by Theorem~\ref{thm:CP}. 
    Equivalently, the coordinate surface has rational self-intersection zero.
\end{proof}

The following obstruction depends only on the real moment-angle complex.
It is especially useful for non-flag spheres with induced missing triangles.
The source of this obstruction is Li's theorem on embedded spheres in symplectic $4$-manifolds~\cite{Li1999}, which states that a symplectic $4$-manifold containing an infinite-order smoothly embedded sphere of nonnegative self-intersection must be rational or ruled.

\begin{corollary}\label{cor:coordinate-sphere-obstruction}
Let $K$ be a simplicial~$3$-sphere. 
Suppose that there is a subset $\omega\subset [m]$ such that $K_\omega\cong \partial P^\ast_3$.
Assume also that $\RZ_K$ is neither rational nor ruled. 
Then $\RZ_K$ is not symplectic.
\end{corollary}
\begin{proof}
    Suppose that $\RZ_K$ is symplectic.
    Since $\RZ_{K_\omega}\cong S^2$, Lemma~\ref{lem:coordinate-surface-dual} gives a smoothly embedded sphere in $\RZ_K$ which represents a nonzero rational homology class and has self-intersection zero.
    Li's theorem~\cite{Li1999} would then imply that $\RZ_K$ is rational or ruled, which is a contradiction.
\end{proof}

We also prepare one obstruction used later.
A closed manifold is called \emph{algebraically simply connected} if its fundamental group has no nontrivial finite quotients.
\begin{proposition}[{\cite[Proposition~1]{Kotschick-Morgan-Taubes1995}}] \label{prop:KMT-connected-sum}
    Let $X$ be a closed symplectic $4$-manifold.
    Suppose that $X$ decomposes as a nontrivial smooth connected sum.
    Then one of the summands has negative definite intersection form and is algebraically simply connected.
\end{proposition}

\section{The flagness obstruction in dimension four}
In this section, we prove the main flagness obstruction in dimension four.
The only exceptional non-flag case for the real moment-angle cover is $\partial(P_3\times P_{m-3})^\ast$, and this case will be excluded later for
small covers by a c-symplectic argument.

Let $K$ be a simplicial~$3$-sphere.
A simplicial complex $K$ is called \emph{flag} if every finite set of vertices which are pairwise joined by edges spans a simplex of $K$.
Equivalently, $K$ is flag if and only if every missing face of $K$ has cardinality two. 
Here a subset~$\omega\subset V(K)$ is called a missing face if $\omega\notin K$ and every proper subset of $\omega$ belongs to $K$.

Since $K$ is a simplicial~$3$-sphere, a non-flag $K$ has a missing face of cardinality~$3$, $4$, or~$5$. 
If $K$ has a missing face of cardinality~$5$, then $K=\partial\Delta^4$.
In this case $\RZ_K\cong S^4$, so $\RZ_K$ is not symplectic.
Moreover, every small cover over $K$ is diffeomorphic to $\R P^4$, hence is non-orientable and not symplectic.
Therefore, it remains to consider the case where $K$ has a missing triangle or a missing tetrahedron.

\subsection{Missing-triangle case}
We first prepare two elementary combinatorial facts.

\begin{lemma} \label{lem:boundary_of_polygon}
    Let $L$ be a finite simplicial complex on the vertex set $V$. 
    Assume that $\widetilde H_1(L;\Q)\cong \Q$ and that $\widetilde H_1(L_\upsilon;\Q)=0$ for every proper subset $\upsilon\subsetneq V$. 
    Then $L$ is the boundary of a polygon.
\end{lemma}

\begin{proof}    
    Choose a simple edge cycle $C$ in $L$ which represents a nonzero class in~$H_1(L;\Q)$.
    Let $V(C)$ be the vertex set of $C$.
    If $V(C)\subsetneq V$, then the class of $C$ is nonzero in~$H_1(L_{V(C)};\Q)$, since otherwise it would also vanish in~$H_1(L;\Q)$.
    This contradicts the assumption. 
    Hence $V(C)=V$.

    The cycle $C$ has no chord.
    Indeed, a chord would split $C$ into two cycles supported on proper subsets of $V$.
    Each of those cycles would bound in the corresponding full subcomplex, and hence $C$ would also bound in $L$, a contradiction.
    
    Thus the $1$-skeleton of $L$ is exactly the cycle $C$. 
    Moreover $L$ has no higher-dimensional simplices. 
    If $|V|>3$, any $2$-simplex would give a chord of $C$.
    If $|V|=3$, filling the triangle would kill $H_1(L;\Q)$.
    Therefore $L=C$, the boundary of a polygon.
\end{proof}

\begin{lemma} \label{lem:missing_triangle}
    Let $K$ be a simplicial~$3$-sphere on~$[m]$.
    Suppose that there exists a subset $\omega\subset [m]$ such that $K_\omega \cong \partial P_3^\ast$.
    Then $b_2(\RZ_K)\ge 2$.
    Moreover, if $b_2(\RZ_K)=2$, then $K \cong \partial (P_3 \times P_{m-3})^\ast$.
\end{lemma}

\begin{proof}
    Recall the Hochster-type decomposition~\eqref{eq:Hochster_decom} for $\RZ_K$
    $$
        H^2(\RZ_K;\Q) \cong \bigoplus_{I\subset [m]}\widetilde H^1(K_I;\Q).
    $$
    Since $K_\omega\cong \partial P_3^\ast \cong S^1$, the summand indexed by $\omega$ contributes one dimension. 
    By Alexander duality for full subcomplexes of the simplicial~$3$-sphere $K$, the complementary full subcomplex $K_{\omega^c}$ also satisfies $\widetilde H^1(K_{\omega^c};\Q)\cong \Q$.
    Hence $b_2(\RZ_K)\ge 2$.

    Assume now that $b_2(\RZ_K)=2$. 
    Then the only nonzero Hochster summands in degree $2$ are those indexed by $\omega$ and ${\omega^c}$. 
    In particular, $\widetilde H_1(K_I;\Q)=0$ for every $I\ne \omega,\omega^c$.
    
    Applying Lemma~\ref{lem:boundary_of_polygon} to $L=K_{\omega^c}$, we obtain $K_{\omega^c}\cong \partial P_q^\ast$ for $q=|{\omega^c}|$.
    
    It remains to see that $K$ is the join of $K_\omega$ and $K_{\omega^c}$. 
    For $x\in \omega^c$, we have $\widetilde H_1(K_{\omega\cup\{x\}};\Q)=0$.
    Since $K_\omega$ is the boundary of a missing triangle, the only possible $2$-simplices in $K_{\omega\cup\{x\}}$ which can bound this cycle are the three triangles containing $x$.
    Thus all these triangles must be present.
    Hence, for every edge $e\subset \omega$, the triangle $e\cup\{x\}$ is a face of $K$.
    
    Therefore, for every edge $e\subset \omega$, the link $\link_K(e)$ is a simplicial $1$-sphere whose vertex set is exactly $\omega^c$. 
    Since $K_{\omega^c}$ is a polygonal cycle, and since every edge of~$\link_K(e)$ is an edge of $K_{\omega^c}$, we must have $\link_K(e)=K_{\omega^c}$.
    Thus $e \cup f$ is a simplex of~$K$ for every edge $e\subset \omega$ and every edge $f\subset \omega^c$ of $K_{\omega^c}$. 
    Equivalently, all facets of $K_\omega \join K_{\omega^c}$ belong to $K$. 
    
    Since $K_\omega$ and $K_{\omega^c}$ are full subcomplexes and both are $1$-dimensional cycles, no additional simplex can occur. 
    Hence $K=K_\omega \join K_{\omega^c} \cong \partial P_3^\ast  \join \partial P_{m-3}^\ast$.
\end{proof}

\begin{lemma} \label{lem:rational-ruled-exclusions}
    The following hold.
    \begin{enumerate}
        \item An orientable four-dimensional small cover is neither rational nor ruled.
        \item Let $K$ be a simplicial~$3$-sphere. 
            If $b_2(\RZ_K)>2$, then $\RZ_K$ is neither rational nor ruled.
    \end{enumerate}
\end{lemma}
\begin{proof}
    (1) Let $M$ be an orientable four-dimensional small cover. 
    By~\eqref{eq:mod2cohom_smallcover}, the ring~$H^\ast(M;\Z_2)$ is generated by $H^1(M;\Z_2)$.
    This is impossible for a rational surface, since then $H^1(-;\Z_2)=0$ but $H^2(-;\Z_2)\neq 0$.
    It is also impossible for a ruled surface over a positive-genus base, since the subring generated by degree-one classes comes from the base and does not contain the fiber class or the exceptional classes.
    
    (2) Since a real moment-angle manifold is stably parallelizable by~\cite{Buchstaber-Panov2015book}, its rational Pontryagin classes vanish.
    Hence, by the Hirzebruch signature theorem, its signature is zero.
    A rational or ruled four-manifold with signature zero has $b_2 = 2$, so the condition $b_2(\RZ_K)>2$ excludes both possibilities.
\end{proof}

\begin{proposition}\label{prop:no-missing-triangle}
    Let $K$ be a simplicial $3$-sphere with a missing triangle.
    If $K$ is not isomorphic to $\partial(P_3\times P_{m-3})^\ast$, then $\RZ_K$ is not symplectic.
\end{proposition}
\begin{proof}
    Since $K$ has a missing triangle, by Lemma~\ref{lem:missing_triangle} together with the assumption, we have $b_2(\RZ_K) \geq 3$.
    Then $\RZ_K$ is neither rational nor ruled by Lemma~\ref{lem:rational-ruled-exclusions}.
    By Corollary~\ref{cor:coordinate-sphere-obstruction}, $\RZ_K$ is not symplectic.
\end{proof}

\subsection{Missing-tetrahedron case}
We recall the connected-sum formula for real moment-angle manifolds under a simplicial connected sum.
\begin{lemma}[\cite{Buchstaber-Panov2015book}]\label{lem:RZ-connected-sum}
    Let $K_1$ and $K_2$ be simplicial $(n-1)$-spheres, and suppose that they contain a common facet $\sigma$ with $|\sigma|=n$.  
    Let $K=K_1\#_\sigma K_2$ be their simplicial connected sum along $\sigma$.  
    Put $a=|V(K_1)\setminus \sigma|$ and $b=|V(K_2)\setminus \sigma|$.
    Then 
    $$
        \RZ_K \cong \left(\#_{2^b}\RZ_{K_1}\right) \# \left(\#_{2^a}\RZ_{K_2}\right) \# \left(\#_{(2^a-1)(2^b-1)} S^1\times S^{n-1}\right).
    $$
\end{lemma}

We now apply this to missing tetrahedra.

\begin{proposition}\label{prop:no-missing-tetrahedron}
    Let $K$ be a simplicial~$3$-sphere with a missing tetrahedron.
    Then $\RZ_K$ is not symplectic.
\end{proposition}
\begin{proof}
    Let $\omega=\{v_1,v_2,v_3,v_4\}\subset V(K)$ be a missing tetrahedron of $K$.
    The subcomplex $K_\omega\cong S^2$ is a simplicial embedded $2$-sphere in $|K|\cong S^3$.  
    Hence it separates $|K|$ into two simplicial $3$-balls.
    Filling $K_\omega$ by the simplex $\Delta_\omega^3$ on each side gives two simplicial~$3$-spheres, say $K_+$ and $K_-$, such that
    $$
        K=K_+\#_\omega K_-.
    $$
    
    Put $A=V(K_+)\setminus \omega$ and $B=V(K_-)\setminus \omega$.
    Both $A$ and $B$ are non-empty.  
    Indeed, if one side had no vertices outside $\omega$, then that side would be a triangulated $3$-ball with boundary $\partial\Delta_\omega^3$ and with no vertices other than those of $\omega$.  
    The only such simplicial $3$-ball is $\Delta_\omega^3$, which would imply $\omega\in K$, contradicting the assumption that $\omega$ is a missing tetrahedron.
    
    Applying Lemma~\ref{lem:RZ-connected-sum} with $n=4$, $a=|A|\geq 1$, and $b=|B|\geq 1$, we obtain 
    $$
        \RZ_K\cong\left(\#_{2^b}\RZ_{K_+}\right) \# \left(\#_{2^a}\RZ_{K_-}\right) \# \left(\#_{(2^a-1)(2^b-1)} S^1\times S^3\right) \cong Y\#(S^1\times S^3),
    $$
    for some closed oriented smooth $4$-manifold $Y$.
    
    Suppose that $\RZ_K$ is symplectic with symplectic form~$\Omega$.
    Then $[\Omega]^2>0$, so $b_2(Y)>0$, since $H^2(S^1\times S^3;\R)=0$.
    By Proposition~\ref{prop:KMT-connected-sum}, one summand in this connected-sum decomposition must be negative definite and algebraically simply connected.
    
    The summand $S^1\times S^3$ is not algebraically simply connected, since $\pi_1(S^1\times S^3)\cong\Z$ has nontrivial finite quotients.
    On the other hand, $Y$ cannot be negative definite, because the symplectic class has positive square and $H^2(\RZ_K;\R)\cong H^2(Y;\R)$.
    This is a contradiction.
\end{proof}

Combining Propositions~\ref{prop:no-missing-triangle} and~\ref{prop:no-missing-tetrahedron}, together with the case $K=\partial\Delta^4$ discussed above, we obtain the following theorem.
\begin{theorem}\label{thm:only_flag_support_symplectic_RZ}
    Let $K$ be a simplicial~$3$-sphere on $[m]$.
    If $\RZ_K$ is symplectic, then $K$ is either isomorphic to $\partial (P_3 \times P_{m-3})^\ast$ or flag.
\end{theorem}

\begin{corollary}\label{cor:only_flag_support_symplectic_small}
    Every symplectic four-dimensional small cover is aspherical.
\end{corollary}
\begin{proof}
    Let $M$ be a symplectic four-dimensional small cover over a simplicial~$3$-sphere $K$.

    By Proposition~\ref{prop:canonical-cover-obstruction}, $\RZ_K$ is symplectic.
    Hence Theorem~\ref{thm:only_flag_support_symplectic_RZ} implies that $K$ is either flag or isomorphic to $\partial(P_3\times P_{m-3})^\ast$.

    The second case is excluded by Proposition~5.2 below, which shows that $\partial(P_3\times P_{m-3})^\ast$ supports no c-symplectic small cover.
    Since every symplectic manifold is c-symplectic, this rules out the exceptional case.

    Therefore $K$ is flag.
    Hence $M$ is aspherical by~\cite{Davis2012}.
\end{proof}

\section{Four-dimensional small covers over products of two polygons} \label{sec-four-dimensional-products}

Throughout this section we consider a simple polytope of the form 
$$
    P=P_{m_1} \times P_{m_2},
$$
where $P_{q}$ is a $q$-gon.
We may assume that $m_1 \leq m_2$.
We write $m=m_1 + m_2$, and decompose the facet set of $P$ as
$$
    \cF(P)=\{F_{1,1},\dots,F_{1,m_1}\} \sqcup \{F_{2,1},\dots,F_{2,m_2}\}
$$
where the block $\{F_{i,1},\dots,F_{i,m_i}\}$ comes from the $i$th polygon factor.

We identify each facet $F_{i,j}$ with the corresponding standard basis vector of $\Z_2^m$.
With this convention, for each $i=1,2$, we set the factor weights
$$
    \chi_i = F_{i,1}+ \dots + F_{i,m_i} \in \Z_2^m.
$$
We also write $\one_m= \sum_{F \in \cF(P)} F = \chi_1 + \chi_2  \in \Z_2^m$.

Let $K=\partial P^\ast$. 
As shown in Example~\ref{ex:RZ_polygon}, 
\begin{equation} \label{eq:RZK_productpolygon}
    \RZ_K \cong \RZ_{\partial P_{m_1}^\ast} \times \RZ_{\partial P_{m_2}^\ast} \cong \Sigma_{\genus(m_1)} \times \Sigma_{\genus(m_2)} =: \Sigma_1 \times \Sigma_2
\end{equation}
where $\Sigma_i=\Sigma_{\genus(m_i)}$ is a closed orientable surface of genus $\genus(m_i)=1+(m_i-4)2^{m_i-3}$.

Let $\lambda\colon \Z_2^m\to \Z_2^4$ be a mod~$2$ characteristic map on $K$, and let 
$$
    M=M(P,\lambda)=M^\R(K,\lambda) = \RZ_K/G_\lambda
$$
be the corresponding small cover, where $G_\lambda = \ker\lambda$.
Moreover, by Corollary~\ref{cor:NN}, $M$ is orientable if and only if $\one_m \in \row(\lambda)$.

We first show that triangle factors cannot occur in c-symplectic small covers over products of polygons.
\begin{lemma}\label{lem:odd-factor-not-compatible}
    Let $P = P_{m_1} \times P_{m_2}$ be a product of two polygons, and let $\lambda$ be a mod~$2$ characteristic map over $P$.
    If  $\chi_i \in \row(\lambda)$, then $m_i$ is even.
\end{lemma}
\begin{proof}
    Since $\chi_i \in \row(\lambda)$, there is a linear functional
    $$
        \varphi\colon \Z_2^{4}\to \Z_2
    $$
    whose values on the columns of $\lambda$ are prescribed by $\chi_i$.
    
    Choose a vertex in the other polygon factor.
    It is given by two adjacent facets.
    Let $W \subset \Z_2^{4}$ be the span of their characteristic vectors.
    By the characteristic condition, $\dim W=2$, and by the definition of $\chi_i$ we have $W\subset \ker\varphi$.
    
    In the quotient $\Z_2^{4}/W$, the images of the characteristic vectors from the $i$th polygon all lie in the affine line
    $$
        \{v\in \Z_2^{4}/W\mid \varphi(v)=1\},
    $$
    which has two elements.
    Adjacent facets of the $i$th polygon have distinct image, again by the characteristic condition.
    Hence the two possible images must alternate around the polygon.
    This is possible only when $m_i$ is even.
\end{proof}

\begin{proposition}\label{prop:triangle-factor-not-csymp}
    Let $P = P_{3} \times P_{m-3}$ be a product of two polygons.
    Then no small cover over $P$ is c-symplectic.
\end{proposition}
\begin{proof}
    Let $K=\partial P^\ast$ and let $M=\RZ_K/G_\lambda$ be a c-symplectic small cover over~$K$.
    By Proposition~\ref{prop-four-manifold-csymp}, there exists $\omega \in \row(\lambda)$ such that $\widetilde{H}^1( K_\omega;\Q) \neq 0$.
    
    Since $K=\partial P_3^\ast\join \partial P_{m-3}^\ast$, every full subcomplex of $K$ is a join of full subcomplexes of the two factors.
    The triangle factor $\partial P_3^\ast$ has no disconnected proper full subcomplex. Hence the only weights $\omega$ with $\widetilde H^1(K_\omega;\Q)\ne0$ are $\chi_1$ and $\chi_2$.
    Thus either $\chi_1\in\row(\lambda)$ or $\chi_2\in\row(\lambda)$.
    Since $\one_m=\chi_1+\chi_2\in\row(\lambda)$, it follows that both $\chi_1$ and $\chi_2$ belong to $\row(\lambda)$.
    
    This contradicts Lemma~\ref{lem:odd-factor-not-compatible}.
    Therefore $M$ is not c-symplectic.
\end{proof}

For $P=P_4\times P_4$, set
$$
    \delta_{1,1}=F_{1,1}+F_{1,3} \quad \text{ and } \quad \delta_{1,2}=F_{1,2}+F_{1,4},
$$
and
$$
    \delta_{2,1}=F_{2,1}+F_{2,3} \quad \text{ and } \quad \delta_{2,2}=F_{2,2}+F_{2,4}.
$$
These are the four opposite-pair vectors.

\begin{definition}
    Fix the product decomposition $P=P_{m_1}\times P_{m_2}$.
    With notation as above, we call the small cover $M(P,\lambda)$ \emph{factor-compatible} if the following two conditions hold.
    \begin{enumerate}
        \item If $(m_1,m_2)=(4,4)$, then the four vectors $\delta_{1,1},\delta_{1,2},\delta_{2,1},\delta_{2,2}$ can be divided into two pairs whose sums both lie in $\row(\lambda)$.
        \item If $(m_1,m_2)\neq(4,4)$, then both $\chi_1$ and $\chi_2$ belong to $\row(\lambda)$.
    \end{enumerate}
\end{definition}

\begin{proposition}\label{prop:factor-compatible-symplectic}
    Let $P=P_{m_1}\times P_{m_2}$ be a product of polygons, and let $M=M(P,\lambda)$ be a factor-compatible small cover.
    Then $M$ is symplectic.
\end{proposition}
\begin{proof}
    If $(m_1,m_2)=(4,4)$, then the factor-compatible condition is exactly the c-symplectic condition in Ishida's theorem, written in the facet notation of the cube. 
    Indeed, for $K=\partial(P_4\times P_4)^\ast=(S^0)^{\join 4}$, the full subcomplexes $K_\omega$ with $\widetilde H^1(K_\omega;\Q)\neq 0$ are precisely those indexed by the sums of two opposite-pair vectors. 
    Hence the above condition is equivalent to c-symplecticity by Proposition~\ref{prop-four-manifold-csymp}.
    By Ishida's theorem~\cite{Ishida2011}., it is also equivalent to symplecticity.
    
    Assume that $(m_1,m_2)\neq(4,4)$.
    Let $K=\partial P^\ast$.
    Under the identification $\RZ_K\cong \Sigma_1\times\Sigma_2$ as in~\eqref{eq:RZK_productpolygon}, the canonical action splits by the two factors.
    Write $\Z_2^m=\Z_2^{m_1}\oplus\Z_2^{m_2}$.
    For $g\in G_\lambda$, write $g=(g^{(1)},g^{(2)})$ according to this decomposition, and let $\rho_i(g)\colon \Sigma_i\to\Sigma_i$ be the diffeomorphism induced by $g^{(i)}$ on the $i$th surface factor.
    
    The orientation character of $\rho_i(g)$ is $(-1)^{\langle \chi_i,g\rangle}$.
    Since $\chi_i\in\row(\lambda)$ and $g\in\ker\lambda$, we have $\langle\chi_i,g\rangle=0$.
    Thus every $\rho_i(g)$ is orientation-preserving.
    
    Choose positive area forms $\eta_1^0$ and $\eta_2^0$ on $\Sigma_1$ and $\Sigma_2$.
    Define
    $$
        \eta_i=\sum_{g\in G_\lambda}\rho_i(g)^\ast\eta_i^0 \qquad (i=1,2).
    $$
    Since every $\rho_i(g)$ is orientation-preserving, each $\eta_i$ is again a positive area form.
    By construction, $\eta_i$ is invariant under the induced $G_\lambda$-action on $\Sigma_i$.
    
    Let $\pi_i \colon \RZ_K\to \Sigma_i$ be the projection. 
    Then 
    $$
        \widetilde\Omega = \pi_1^\ast\eta_1+\pi_2^\ast \eta_2
    $$
    is closed and $G_\lambda$-invariant. 
    Moreover, $\widetilde\Omega^2 = 2\pi_1^\ast\eta_1\wedge \pi_2^\ast\eta_2$ is a positive volume form on $\Sigma_1 \times \Sigma_2$.
    Hence $\widetilde\Omega$ is a $G_\lambda$-invariant symplectic form on $\RZ_K$, and it descends to a symplectic form on $M$ by Proposition~\ref{prop:canonical-cover-obstruction}.
\end{proof}

We now turn to the converse direction.
From now on, we prove that every symplectic small cover over a product of two polygons is factor-compatible.
We introduce the geometric obstruction that controls symplectic small covers over products of two polygons.
\begin{proposition}\label{prop-product-factor-reversal}
    Let $P=P_{m_1}\times P_{m_2}$ be a product of two polygons with $m_1,m_2\geq 4$ and $(m_1,m_2)\ne (4,4)$, and put $K=\partial P^\ast$.
    Let $M=\RZ_K/G_\lambda$ be an oriented small cover.
    If some element of $G_\lambda$ reverses the orientations of both surface factors of~$\RZ_K$, then $M$ is not symplectic.
\end{proposition}
\begin{proof}
    Suppose that $M$ is symplectic.
    Let $\Omega$ be a symplectic form on $M$.
    Let $\pi \colon \RZ_K\to M$ be the finite covering map, and let $\widetilde\Omega=\pi^\ast\Omega$ be the pullback symplectic form.
    Every element of $G_\lambda$ acts as a symplectomorphism of $\RZ_K\cong \Sigma_1\times \Sigma_2$ with respect to $\widetilde\Omega$.
    
    Choose an almost complex structure $J$ compatible with $\widetilde\Omega$, and let
    $$
    K_{\widetilde\Omega}:=-c_1(T\RZ_K,J)
    $$
    be the symplectic canonical class of $\RZ_K$.
    Since $(m_1,m_2)\neq (4,4)$, one of $\genus(m_1)$ and $\genus(m_2)$ is greater than $1$.
    Hence 
    \begin{equation} \label{eq:b2+}
        b_2^+(\RZ_K)=1+2\genus(m_1)\genus(m_2)>3.      
    \end{equation}

    Let $\gamma_1\subset \Sigma_1$ and $\gamma_2\subset \Sigma_2$ be embedded loops representing nonzero integral homology classes.
    Then
    $$
        T_{\gamma_1,\gamma_2}:=\gamma_1\times\gamma_2
    $$
    is an embedded torus in $\RZ_K$.
    Its homology class is non-torsion, and its self-intersection is zero.

    By Taubes' non-vanishing theorem and the standard Seiberg--Witten adjunction inequality for $b_2^+>1$~\cite{Taubes1994, Ozsvath-Szabo2000}, applied to the symplectic canonical class $K_{\widetilde\Omega}$, we have
    $$
        0=2\genus(T_{\gamma_1,\gamma_2})-2 \ge [T_{\gamma_1,\gamma_2}]^2+\left|K_{\widetilde\Omega}\cdot [T_{\gamma_1,\gamma_2}]\right|= \left|K_{\widetilde\Omega}\cdot [T_{\gamma_1,\gamma_2}]\right|.
    $$
    Therefore $K_{\widetilde\Omega}\cdot [T_{\gamma_1,\gamma_2}]=0$ for every product torus of this form.

    As $\gamma_1$ and $\gamma_2$ vary, these product tori span the mixed K\"unneth summand
    $$
        H_1(\Sigma_1;\Q)\otimes H_1(\Sigma_2;\Q) \subset H_2(\RZ_K;\Q).
    $$
    Hence, by K\"unneth duality, the component of $K_{\widetilde\Omega}$ in $H^1(\Sigma_1;\Q)\otimes H^1(\Sigma_2;\Q)$ vanishes.
    Thus the rational canonical class lies in $\Q u_1\oplus \Q u_2$, where $u_1$ and $u_2$ are the orientation classes of the two surface factors.
    
    Let $g \in G_\lambda$ be an element reversing the orientations of both surface factors.
    Then $g$ acts by multiplication by $-1$ on each of the two factor orientation classes.
    Therefore
    $$
        g^\ast K_{\widetilde\Omega}=-K_{\widetilde\Omega}.
    $$
    On the other hand, $g$ is a symplectomorphism of the pulled-back symplectic form, so it preserves the symplectic canonical class.
    Thus
    $$
        g^\ast K_{\widetilde\Omega}=K_{\widetilde\Omega}.
    $$
    It follows that $K_{\widetilde\Omega}=0$ in $H^2(\RZ_K;\Q)$.
    Since $\RZ_K$ is a product of surfaces, its integral cohomology is torsion-free. 
    Thus $K_{\widetilde\Omega}=0$ in $H^2(\RZ_K;\Z)$.    
    Equivalently, $c_1(T\RZ_K,J)=0$.
    In particular, $c_1(T\RZ_K,J)$ is torsion.
    By Bauer's theorem~\cite[Corollary~1.2]{Bauer2008}, applied to the closed symplectic four-manifold $(\RZ_K,\widetilde\Omega)$, we have
    $$
        b_2^+(\RZ_K)\le 3.
    $$
    This contradicts~\eqref{eq:b2+}.
    Hence $M$ is not symplectic.
\end{proof}

We now apply the obstruction to the failure of factor-compatibility.
For a product of two polygons, write $\chi_1,\chi_2\in \Z_2^m$ for the vectors supported on the facets of the two factors.
When at least one factor is not a square, factor-compatibility is equivalent to the condition that both $\chi_1$ and $\chi_2$ lie in $\row(\lambda)$.
When both factors are squares, this is the real Bott case.

\begin{lemma}\label{lem-failure-factor-compatible-reverses}
    Let $P=P_{m_1}\times P_{m_2}$ with $(m_1,m_2) \neq (4,4)$ and let $M=M(P,\lambda)$ be orientable.
    Assume that $M$ is not factor-compatible.
    Then there exists $g \in \ker\lambda$ which reverses the orientations of both surface factors of $\RZ_{\partial P^\ast}\cong\Sigma_1\times\Sigma_2$.
\end{lemma}
\begin{proof}
    Since $M$ is orientable, Corollary~\ref{cor:NN} gives $\one_{m}=\chi_1+\chi_2\in\row(\lambda)$.
    By the assumption, one of $\chi_1,\chi_2$ does not lie in $\row(\lambda)$.
    
    Without loss of generality, assume that $\chi_1\notin\row(\lambda)$.
    By Lemma~\ref{lem:row-kernel-duality}, there exists $g \in\ker\lambda$ such that $\langle \chi_1, g \rangle=1$.
    Since
    $$
        0 = \langle \one_{m},g \rangle = \langle \chi_1+\chi_2,g\rangle = 1+\langle \chi_2,g\rangle,
    $$
    we have $\langle \chi_2,g\rangle=1$.
    The orientation character of the action of $g$ on the $i$th surface factor is $(-1)^{\langle\chi_i,g\rangle}$.
    Therefore $\langle\chi_i,g\rangle=1$ means that $g$ reverses the orientation of the $i$th surface factor.
\end{proof}

Combining the previous results gives the four-dimensional classification for products of two polygons.

\begin{theorem}\label{thm-two-polygons-factor-compatible}
    Let $P$ be a product of two polygons.
    A small cover over $P$ is symplectic if and only if it is factor-compatible.
\end{theorem}
\begin{proof}
    Let $M$ be a small cover over $P=P_{m_1} \times P_{m_2}$.
    If one of $m_1$, $m_2$ is equal to $3$, then Proposition~\ref{prop:triangle-factor-not-csymp} shows that no small cover over $P$ is c-symplectic, hence none is symplectic. 
    On the other hand, Lemma~\ref{lem:odd-factor-not-compatible} shows that no such small cover is factor-compatible. 
    Thus the equivalence holds in this case. 
    We may therefore assume $m_1,m_2 \geq 4$.
    
    If the small cover is factor-compatible, then it is symplectic by Proposition~\ref{prop:factor-compatible-symplectic}.
    
    Conversely, assume that $M$ is symplectic.
    If $(m_1,m_2)=(4,4)$, then $M$ is a real Bott manifold.
    By~\cite{Ishida2011}, c-symplecticity and symplecticity are equivalent for real Bott manifolds.
    Since factor-compatibility is defined as c-symplecticity in this case, the claim follows.
    
    Now assume $(m_1,m_2)\ne(4,4)$.
    If $M$ is not factor-compatible, Lemma~\ref{lem-failure-factor-compatible-reverses} would give an element of $\ker\lambda$ reversing both surface factors.
    Proposition~\ref{prop-product-factor-reversal} would then rule out symplectic structures on $M$, a contradiction.
    Hence $M$ is factor-compatible.
\end{proof}

\begin{example}\label{ex-csymp-not-factor-compatible-P5P4}
    Let $P=P_5\times P_4$.
    Consider the characteristic matrix
    $$
        \lambda=
            \begin{pNiceArray}{ccccc|cccc}[first-row,first-col]
                &
                \scriptstyle F_{1,1}&\scriptstyle F_{1,2}&\scriptstyle F_{1,3}
                &\scriptstyle F_{1,4}&\scriptstyle F_{1,5}
                &
                \scriptstyle F_{2,1}&\scriptstyle F_{2,2}&\scriptstyle F_{2,3}&\scriptstyle F_{2,4}
                \\
                \lambda_1&
                1&0&1&0&1&
                0&0&0&0
                \\
                \lambda_2&
                0&1&0&1&1&
                0&0&0&0
                \\
                \lambda_3&
                0&0&0&0&1&
                1&0&1&0
                \\
                \lambda_4&
                0&0&0&0&0&
                0&1&0&1
            \end{pNiceArray}.
    $$
    It is straightforward to check the characteristic condition.
    
    Since $\lambda_1 + \cdots + \lambda_4 =\one_9 \in \row(\lambda)$, the corresponding small cover $M=M(P,\lambda)$ is orientable by Corollary~\ref{cor:NN}.
    Moreover, the row space contains 
    $$
        \omega = F_{1,1}+F_{1,3}+F_{2,1}+F_{2,3} = \lambda_1 + \lambda_3.
    $$
    The induced full subcomplex is $K_\omega \cong S^0\ast S^0 \cong S^1$, and hence $\widetilde H^1(K_\omega;\Q)\cong \Q$.
    Therefore, by Proposition~\ref{prop-four-manifold-csymp}, $M$ is c-symplectic.
    
    On the other hand, since the first factor has five sides, Lemma~\ref{lem:odd-factor-not-compatible} gives $\chi_1\notin \row(\lambda)$.
    Hence $M$ is not factor-compatible.
    Therefore Theorem~\ref{thm-two-polygons-factor-compatible} implies that $M$ is not symplectic.
    
    This example shows that the equivalence between c-symplecticity and symplecticity for real Bott manifolds~\cite{Ishida2011} does not extend to small covers over products of two polygons. 
\end{example}

\section{Smooth classification of factor-compatible small covers}\label{sec:two-stage-rigidity}

In this section, we discuss the smooth type of a symplectic small covers $M$ over a product $P = P_{m_1}\times P_{m_2}$ of two polygons.
First, we shall classify $M$ up to D--J equivalence.

For an even integer $m$, set
$$
    \one_m=(1,\ldots,1) \quad \text{ and } \quad \varepsilon_m=(1,0,1,0,\ldots,1,0)\in \Z_2^{m}.
$$

\begin{lemma}\label{lem:two-stage-normal-form}
    Every two-stage factor-compatible small cover is, after permuting the two polygon factors, and up to
    D--J equivalence, represented by a characteristic matrix of the form
    \begin{equation}\label{eq:two-stage-normal-form}
        \lambda_\beta=
        \begin{pmatrix}
            \one_{m_1} & 0\\
            \varepsilon_{m_1} & 0\\
            0 & \one_{m_2}\\
            \beta & \varepsilon_{m_2}
        \end{pmatrix},
    \end{equation}
    where $\beta\in \Z_2^{m_1}$.
\end{lemma}
\begin{proof}
    By factor-compatibility, after elementary row operations, we may assume that the first and third rows are the factor weights $\chi_1$ and $\chi_2$, respectively. 
    With respect to this choice of row basis, the second and fourth rows can be written as $\begin{pmatrix} \varepsilon_{m_1} & \beta_{m_2}\end{pmatrix}$ and $\begin{pmatrix}  \beta_{m_1} &\varepsilon_{m_2} \end{pmatrix}$, respectively.
    Hence,
    $$
        \lambda=
        \begin{pmatrix}
            \one_{m_1} & 0\\
            \varepsilon_{m_1} & \beta_{m_2}\\
            0 & \one_{m_2}\\
            \beta_{m_1} & \varepsilon_{m_2}
        \end{pmatrix}.
    $$
    
    Suppose both $\beta_{m_2}$ and $\beta_{m_1}$ are nonzero. 
    
    For adjacent pairs in $P_{m_1}$ and $P_{m_2}$, set
    $$
        A_i=\beta_{m_1}(F_{1,i})+\beta_{m_1}(F_{1,i+1}) \qquad B_j=\beta_{m_2}(F_{2,j})+\beta_{m_2}(F_{2,j+1}).
    $$
    At the vertex determined by $F_{1,i}\cap F_{1,i+1}\cap F_{2,j}\cap F_{2,j+1}$, the four corresponding columns must form a basis of $\Z_2^4$. 
    After adding the first of the two $P_{m_1}$-columns to the second, and the first of the two $P_{m_2}$-columns to the second, the determinant becomes
    $$
        \det
        \begin{pmatrix}
            1 & B_j\\
            A_i & 1
        \end{pmatrix}
        = 1+A_iB_j.
    $$
    Hence the characteristic condition forces $A_iB_j=0$ for all $i,j$.
    Therefore either $A_i=0$ for all $i$, or $B_j=0$ for all $j$.
\end{proof}

\begin{corollary}\label{cor:number-two-stage-symplectic}
    Let $m_1$ and $m_2$ be even integers with $m_1,m_2\ge 4$, and put $P=P_{m_1}\times P_{m_2}$.
    Up to D--J equivalence, the number of symplectic small covers over $P$ is
    $$
        2^{m_1-2}+2^{m_2-2}-1.
    $$
\end{corollary}

\begin{proof}
    A small cover over $P$ is symplectic if and only if it is factor-compatible by Theorem~\ref{thm-two-polygons-factor-compatible}.
    By Lemma~\ref{lem:two-stage-normal-form}, the only off-diagonal datum is a class in $[\beta]\in \Z_2^{m_i}/\langle \one_{m_i},\varepsilon_{m_i}\rangle$, so each direction contributes $2^{m_i-2}-1$ nonzero classes, while the zero class gives the unique product case.
\end{proof}

Since each diagonal block is an orientable characteristic matrix over a polygon, the diagonal small covers over $P_{m_1}$ and $P_{m_2}$ are the orientable surfaces 
$$
    \widetilde{\Sigma}_1=\Sigma_{\genus_{sc}(m_1)} \quad \text{ and } \quad \widetilde{\Sigma}_2=\Sigma_{\genus_{sc}(m_2)},
$$
respectively, where $\genus_{sc}(m)=\frac{m-2}{2}$.

On the other hand, Since $H^1(\widetilde{\Sigma}_1;\Z_2)\cong\Z_2^{m_1}/\langle \one_{m_1},\varepsilon_{m_1}\rangle$, the class $[\beta]$ may be regarded as an element $\alpha\in H^1(\widetilde{\Sigma}_1;\Z_2)$, where
$$
    \alpha=[\beta]\in H^1(\widetilde{\Sigma}_1;\Z_2).
$$
We denote by $M_\alpha$ the small cover represented by \eqref{eq:two-stage-normal-form}.

\begin{proposition}\label{prop:two-stage-diffeomorphism-classification}
    For fixed $m_1$ and $m_2$, the diffeomorphism type of $M_\alpha$ depends only on whether $\alpha$ is zero or nonzero. 
    More precisely,
    $$
        M_\alpha\cong  \widetilde{\Sigma}_1\times \widetilde{\Sigma}_2 \quad\text{if } \alpha=0,
    $$
    whereas all small covers $M_\alpha$ with $\alpha\ne0$ are mutually diffeomorphic.
\end{proposition}

\begin{proof}
    If $\alpha=0$, then $\lambda_\beta$ is D--J equivalent to the block diagonal matrix.
    Thus $M_\alpha\cong \widetilde{\Sigma}_1\times \widetilde{\Sigma}_2$.
    
    Now suppose that $\alpha\ne0$.
    The projection $P_{m_1}\times P_{m_2}\to P_{m_1}$ induces an equivariant bundle
    $$
        \widetilde{\Sigma}_2\longrightarrow M_\alpha\longrightarrow \widetilde{\Sigma}_1.
    $$
    The off-diagonal row $\beta$ records the monodromy of this bundle. 
    More precisely, $M_\alpha$ is the associated bundle
    $$
        M_\alpha\cong L_\alpha\times_{\Z_2}\widetilde{\Sigma}_2,
    $$
    where $L_\alpha\to \widetilde{\Sigma}_1$ is the principal $\Z_2$-bundle classified by $\alpha\in H^1(\widetilde{\Sigma}_1;\Z_2)$, and the nontrivial element of $\Z_2$ acts on the fiber $\widetilde{\Sigma}_2$ by the fixed involution induced by the second generator of the standard small cover over $P_{m_2}$.
    
    It remains to show that all nonzero $\alpha$ give diffeomorphic total spaces.
    Let $\alpha,\alpha'\in H^1(\widetilde{\Sigma}_1;\Z_2)$ be nonzero. 
    The mapping class group of $\widetilde{\Sigma}_1$ acts transitively on the nonzero elements of $H^1(\widetilde{\Sigma}_1;\Z_2)$. 
    Indeed, by Poincar\'e duality, nonzero cohomology classes correspond to nonzero mod~$2$ homology classes, and such classes can be represented by nonseparating simple closed curves. 
    The mapping class group is transitive on nonseparating simple closed curves.
    
    Hence there exists a diffeomorphism $f\colon \widetilde{\Sigma}_1\to \widetilde{\Sigma}_1$ such that $f^\ast\alpha'=\alpha$.
    Therefore the principal $\Z_2$-bundle $L_\alpha$ is isomorphic to $f^\ast L_{\alpha'}$. 
    Taking the associated bundle with the same fiber action on $\widetilde{\Sigma}_2$ gives a diffeomorphism
    $$
        L_\alpha\times_{\Z_2}\widetilde{\Sigma}_2 \cong L_{\alpha'}\times_{\Z_2}\widetilde{\Sigma}_2.
    $$
    Thus $M_\alpha\cong M_{\alpha'}$ for all nonzero $\alpha,\alpha'$.
\end{proof}

We now show that the mod~$2$ cohomology ring detects the two cases in Proposition~\ref{prop:two-stage-diffeomorphism-classification}.

\begin{lemma}\label{lem:two-stage-square-rank}
    Let $R_\alpha=H^\ast(M_\alpha;\Z_2)$, and let $q_\alpha \colon R_\alpha^1\to R_\alpha^2$ be the square map on degree-one classes, that is, $q_\alpha(z)=z^2$. Then
    $$
        \rank q_\alpha = \begin{cases} 
            0, & \alpha=0,\\
            m_2 -2, & \alpha\ne0.
        \end{cases}
    $$
\end{lemma}

\begin{proof}
    By~\eqref{eq:mod2cohom_smallcover}, one can see that $u^2=0$ for every $u\in H^1(\widetilde\Sigma_1;\Z_2)$, and $v^2= \alpha v$ for every $v\in H^1(\widetilde\Sigma_2;\Z_2)$.
    
    Since the coefficient field is $\Z_2$, the square map on degree-one classes is linear. 
    If $z=u+v\in R_\alpha^1$, then
    $$
        z^2= \alpha v.
    $$
    Consequently, if $\alpha=0$, then $q_\alpha=0$.
    
    Assume now that $\alpha\ne0$. The image of $q_\alpha$ is $\alpha\cdot H^1(\widetilde\Sigma_2;\Z_2)$.
    Apply again~\eqref{eq:mod2cohom_smallcover}, one can see that multiplication by $\alpha$ is injective on the fiber degree-one part, as desired.
\end{proof}

\begin{theorem}\label{thm:two-stage-mod2-rigidity}
    Let $M$ and $M'$ be symplectic small covers over products of two polygons. 
    If
    $$
        H^\ast(M;\Z_2)\cong H^\ast(M';\Z_2)
    $$
    as graded rings, then $M$ and $M'$ are diffeomorphic.
\end{theorem}

\begin{proof}
    Since products of two polgyons are distinct by their face numbers, they are decided by $H^\ast(M;\Z_2)$ by \eqref{eq:mod2betti}.
    Put $P =P_{m_1} \times P_{m_2}$.
    By Lemma~\ref{lem:two-stage-normal-form}, we may write $M=M_\alpha$ for some $\alpha\in H^1(\widetilde\Sigma_1;\Z_2)$, where $M_\alpha$ is represented by the matrix~\eqref{eq:two-stage-normal-form}. 
    Then, the diffeomorphism type is determined by the nonzeroness of $\alpha$ by Proposition~\ref{prop:two-stage-diffeomorphism-classification}, and it is detected by the $H^\ast(M_\alpha;\Z_2)$ by Lemma~\ref{lem:two-stage-square-rank}.
    If $\alpha \neq 0$, the rank $q$ only depends on $m_2$. 
    Hence, one can even decide the which factor is base.
    
    By \eqref{eq:mod2betti}, the graded mod~$2$ cohomology ring determines the $h$-vector of the orbit polytope. 
    For a product of two polygons $P_{m_1}\times P_{m_2}$, this $h$-vector determines the unordered pair $\{m_1,m_2\}$. 
    Hence, after possibly interchanging the two factors, we may assume that both $M$ and $M'$ lie over the same product $P=P_{m_1} \times P_{m_2}$.
    
    By Lemma~\ref{lem:two-stage-normal-form}, after possibly interchanging the two factors, we may write $M=M_\alpha$ for some $\alpha\in H^1(\widetilde\Sigma_1;\Z_2)$, where $M_\alpha$ is represented by the normal form \eqref{eq:two-stage-normal-form}. 
    Similarly, write $M'=M_{\alpha'}$.
    
    By Proposition~\ref{prop:two-stage-diffeomorphism-classification}, for fixed $m_1$ and $m_2$, the diffeomorphism type of $M_\alpha$ depends only on whether
    $\alpha$ is zero or nonzero. 
    By Lemma~\ref{lem:two-stage-square-rank}, this condition is detected by the graded ring $H^\ast(M_\alpha;\Z_2)$, namely by the rank of the square map on degree-one classes.
    Thus $\alpha=0$ if and only if $\alpha'=0$.
    
    If $\alpha=0$, then both manifolds are diffeomorphic to $\widetilde\Sigma_1\times \widetilde\Sigma_2$.
    If $\alpha\ne0$, then Lemma~\ref{lem:two-stage-square-rank} gives $\rank(x\mapsto x^2)=m_{\mathrm{fiber}}-2$ so the graded ring also determines which factor is the base, except in the case $m_1=m_2$, where the distinction is irrelevant. 
    Proposition~\ref{prop:two-stage-diffeomorphism-classification} then implies that all nonzero choices of $\alpha$ give diffeomorphic total spaces.
    Therefore $M$ and $M'$ are diffeomorphic.
\end{proof}

The property that the diffeomorphism type of a small cover is determined by its mod~$2$ cohomology ring is called a \emph{mod~$2$ cohomology rigidity}.
It does not hold for general small covers~\cite{Masuda2010}, but it holds for some specific cases such as real Bott manifolds~\cite{Kamishima-Masuda2009,Choi-Masuda-Oum2017}.
See~\cite{Choi-Masuda-Suh2011} for details.

\begin{corollary} \label{cor:number_of_diffeo_types}
    Let $m_1$ and $m_2$ be even integers with $m_1,m_2\ge 4$, and put $P=P_{m_1}\times P_{m_2}$.
    Up to diffeomorphism, the number of symplectic small covers over $P$ is
    $$
        \begin{cases}
            2, & m_1=m_2,\\
            3, & m_1\ne m_2.
        \end{cases}
    $$
\end{corollary}
\begin{proof}
    The zero off-diagonal class gives the product type, and for each ordered base-fiber choice all nonzero classes give one non-product type.
    Two choices are identified when $m_1=m_2$, while for $m_1\ne m_2$ they are distinguished by $\rank(x\mapsto x^2)$.
\end{proof}

\section{Symplectic small covers beyond products of two polygons}

In the previous sections, we proved that the flag condition is an important obstruction for four-dimensional small covers and real moment-angle complexes.
At present, however, we do not know any flag simplicial~$3$-sphere $K$ for which $\RZ_K$ is not symplectic.
This suggests that flagness may be close to the right combinatorial condition.

On the other hand, explicit symplectic examples are still quite limited.
So far, the most accessible family comes from products of two polygons.
It is not clear a priori whether there exist symplectic small covers over flag spheres which are not dual to products of two polygons.
The goal of this section is to show that such examples do exist.

Our main tool is the following criterion for products with a circle.
For a closed oriented $3$-manifold $N$, the manifold $S^1\times N$ is symplectic if and only if $N$ fibers over $S^1$.
The ``if'' direction goes back to Thurston~\cite{Thurston1976}, and the converse was proved by Friedl and Vidussi~\cite{Friedl-Vidussi2011}.

This leads us naturally to the case 
$$
    K=\partial(I\times Q)^\ast,
$$
where $Q$ is a simple $3$-polytope.
Indeed, $\partial(I\times Q)^\ast=\partial I^\ast \join \partial Q^\ast$, and hence
$$
    \RZ_K \cong \RZ_{\partial I^\ast}\times \RZ_{\partial Q^\ast} \cong S^1\times \RZ_{\partial Q^\ast}.
$$
Therefore, if $\RZ_{\partial Q^\ast}$ fibers over $S^1$, then $\RZ_K$ is symplectic.

If $Q$ is decomposable, then since $\dim Q=3$, it must be of the form $I\times P_m$ for some polygon $P_m$.
In that case, $I\times Q \cong P_4\times P_m$, which is again a product of two polygons.
Thus, in order to find genuinely new examples, it is enough to consider the case where $Q$ is indecomposable.

We begin with a brief review of the general background on fibrations of $3$-manifolds over the circle.
After that, we introduce a specific indecomposable simple $3$-polytope $Q$, let $L=\partial Q^\ast$, and construct a symplectic small cover over~$I\times Q$.

We recall a standard way to prove that a closed $3$-manifold fibers over the circle.
Let $N$ be a closed orientable irreducible $3$-manifold.
By Stallings' fibering theorem~\cite{Stallings1961}, if a primitive class $\phi\in H^1(N;\Z)$ has finitely generated kernel $\ker\phi\subset \pi_1(N)$, then $\phi$ is represented by a fibration 
$$
    N\to S^1.
$$
Thus, in our setting, it is enough to construct such a class and prove that its kernel is finitely generated.

We shall use Bestvina--Brady Morse theory~\cite{Bestvina-Brady1997} for this purpose.
Let $L$ be a flag simplicial~$2$-sphere on the vertex set $[m]$, and let $\mu\colon \Z_2^m\to \Z_2^3$ be a mod~$2$ characteristic map over~$L$.
Write $\mu_i:=\mu(e_i)\in \Z_2^3$ for $i\in [m]$.
Let $N=M^\R(L,\mu)$ be the corresponding orientable small cover.
We use the cubewise cell structure on~$N$.
Its vertices are indexed by elements of~$\Z_2^3$, and the oriented edge in the $i$th direction at a vertex $g\in\Z_2^3$ goes from~$g$ to~$g+\mu_i$.

Let $c$ be an integral cubewise $1$-cocycle on $N$.
Assume that $c$ is nonzero on every oriented edge and is affine on every square.
Equivalently, opposite parallel edges in each square have the same value under $c$.
Then $c$ lifts to a cubewise Morse function on the universal cover of $N$ and defines a homomorphism
$$
    [c]\colon \pi_1(N)\to \Z.
$$

For each vertex $g\in\Z_2^3$, define
$$
    P_g=\{i\mid c(g,i)>0\} \quad \text{ and } \quad N_g=\{i\mid c(g,i)<0\}.
$$
The ascending and descending links at $g$ are the induced subcomplexes~$L_{P_g}$ and~$L_{N_g}$, respectively.

The following criterion is a standard consequence of Stallings' fibering theorem and Bestvina--Brady Morse theory.
\begin{proposition}\label{prop:fibering-criterion-small-cover}
    Let $L$ be a flag simplicial~$2$-sphere, and let
    $N=M^\R(L,\mu)$ be an orientable small cover.
    Suppose that there exists an integral cubewise $1$-cocycle~$c$ satisfying
    the following conditions.
    \begin{enumerate}
        \item $c$ is nonzero on every oriented edge.
        \item $c$ is affine on every square.
        \item For every vertex $g\in\Z_2^3$, the induced subcomplexes $L_{P_g}$ and $L_{N_g}$ are nonempty and connected.
    \end{enumerate}
    If the image of $[c]\colon\pi_1(N)\to\Z$ is $d\Z$ for some $d>0$, then the primitive class
    $$
        \phi=\frac{1}{d}[c]\in H^1(N;\Z)
    $$
    is represented by a fibration $N\to S^1$.
\end{proposition}

Indeed, the affine cocycle gives a cubewise Morse function on the universal cover. 
The connectivity assumption on the ascending and descending links implies, by Bestvina--Brady Morse theory~\cite{Bestvina-Brady1997}, that $\ker[c]$ is finitely generated. 
Since $L$ is flag, $N$ is aspherical and hence irreducible, so Stallings' theorem~\cite{Stallings1961} applies to the primitive class $\phi$.

\subsection{A symplectic small cover example over $I\times Q$}
We now describe our example.

Let $L$ be the flag simplicial complex on the vertex set $\{0,\ldots, 9\}$ whose facets are
$$
    \begin{aligned}
        &(0,6,8), (0,6,9), (0,7,8), (0,7,9), (1,2,5), (1,2,6), (1,3,5), (1,3,7),\\
        &(1,6,8), (1,7,8), (2,4,5), (2,4,6), (3,4,5), (3,4,7), (4,6,9), (4,7,9).
    \end{aligned}
$$
Let $Q$ be the simple $3$-polytope dual to $L$, so that $L=\partial Q^\ast$.
See Figure~\ref{fig:Q_L}
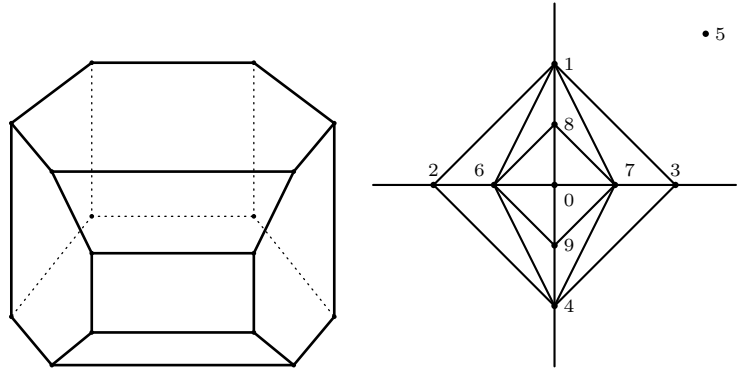
\begin{figure}
    \centering
    \begin{tikzpicture}
    [
        scale=0.8,
        line join=round,
        line cap=round,
        visible/.style={draw=black, line width=1pt},
        hidden/.style={draw=black, dotted, line width=0.5pt}
    ]
    
    \coordinate (Aone)   at (-2.00, 3.20);
    \coordinate (Bone)   at ( 2.00, 3.20);
    \coordinate (Cone)   at ( 2.67, 4);
    \coordinate (Done)   at ( 1.34, 5);
    \coordinate (Eone)   at (-1.34, 5);
    \coordinate (Fone)   at (-2.67, 4);
    
    \coordinate (Atwo)   at (-2.00, 0.00);
    \coordinate (Btwo)   at ( 2.00, 0.00);
    \coordinate (Ctwo)   at ( 2.67, 0.80);
    \coordinate (Dtwo)   at ( 1.34, 2.46);
    \coordinate (Etwo)   at (-1.34, 2.46);
    \coordinate (Ftwo)   at (-2.67, 0.80);
    
    \coordinate (Athree) at (-1.34, 1.85);
    \coordinate (Bthree) at ( 1.34, 1.85);
    \coordinate (Afour)  at (-1.34, 0.54);
    \coordinate (Bfour)  at ( 1.34, 0.54);
    
    \draw[visible] (Aone)--(Bone)--(Cone)--(Done)--(Eone)--(Fone)--cycle;
    
    \draw[visible] (Fone)--(Ftwo);
    \draw[visible] (Cone)--(Ctwo);
    \draw[visible] (Ftwo)--(Atwo)--(Btwo)--(Ctwo);
    
    \draw[visible] (Aone)--(Athree)--(Afour)--(Atwo);
    \draw[visible] (Bone)--(Bthree)--(Bfour)--(Btwo);
    \draw[visible] (Athree)--(Bthree);
    \draw[visible] (Afour)--(Bfour);
    
    \draw[hidden] (Done)--(Dtwo);
    \draw[hidden] (Eone)--(Etwo);
    \draw[hidden] (Etwo)--(Dtwo);
    \draw[hidden] (Dtwo)--(Ctwo);
    \draw[hidden] (Etwo)--(Ftwo);
    
    \foreach \P in {Aone,Bone,Cone,Done,Eone,Fone,Atwo,Btwo,Ctwo,Dtwo,Etwo,Ftwo,Athree,Bthree,Afour,Bfour}
        \fill (\P) circle (1.1pt);
    
    \end{tikzpicture}
    \quad
    \begin{tikzpicture}[scale=0.8,line join=round,line cap=round]

\coordinate (v5) at ( 2.5, 2.5);

\coordinate (v1) at ( 0,  2);
\coordinate (v8) at ( 0,  1);
\coordinate (v0) at ( 0,  0);
\coordinate (v9) at ( 0, -1);
\coordinate (v4) at ( 0, -2);

\coordinate (v2) at ( -2,  0);
\coordinate (v6) at ( -1,  0);

\coordinate (v3) at (2, 0);
\coordinate (v7) at (1, 0);

\coordinate (v5u) at ( 0, 3);
\coordinate (v5d) at (0, -3);
\coordinate (v5r) at ( 3, 0);
\coordinate (v5l) at (-3, 0);

\draw[thick] (v0)--(v6);
\draw[thick] (v0)--(v8);
\draw[thick] (v0)--(v9);
\draw[thick] (v0)--(v7);

\draw[thick] (v1)--(v2);
\draw[thick] (v1)--(v3);
\draw[thick] (v1)--(v6);
\draw[thick] (v1)--(v8);
\draw[thick] (v1)--(v5u);
\draw[thick] (v1)--(v7);

\draw[thick] (v2)--(v5l);
\draw[thick] (v2)--(v6);
\draw[thick] (v2)--(v4);

\draw[thick] (v3)--(v4);
\draw[thick] (v3)--(v5r);
\draw[thick] (v3)--(v7);

\draw[thick] (v4)--(v5d);
\draw[thick] (v4)--(v6);
\draw[thick] (v4)--(v7);
\draw[thick] (v4)--(v9);

\draw[thick] (v6)--(v8);
\draw[thick] (v6)--(v9);

\draw[thick] (v7)--(v8);
\draw[thick] (v7)--(v9);

\foreach \i in {0,1,2,3,4,5, 6,7,8,9}{
    \fill (v\i) circle (1.5pt);
}

\node[right] at (v1) {$\scriptstyle 1$};
\node[above left] at (v6) {$\scriptstyle 6$};
\node[above] at (v2) {$\scriptstyle 2$};
\node[right] at (v4) {$\scriptstyle 4$};
\node[right] at (v5) {$\scriptstyle 5$};
\node[above] at (v3) {$\scriptstyle 3$};
\node[above right] at (v7) {$\scriptstyle 7$};

\node[right] at (v8) {$\scriptstyle 8$};
\node[below right] at (v0) {$\scriptstyle 0$};
\node[right] at (v9) {$\scriptstyle 9$};

\end{tikzpicture}
    \caption{The polytope $Q$ and a Schlegel diagram of its dual complex $L$. 
    In the diagram of $L$, the vertex $5$ is placed at infinity.}
    \label{fig:Q_L}
\end{figure}

The polytope $Q$ is indecomposable, so $I\times Q$ is not a product of two polygons.
It has six square faces, and four hexagonal faces.

Let $L=\partial Q^\ast$.
Define a linear map $\mu\colon \Z_2^{10}\to \Z_2^3$ by the characteristic matrix
$$
    \mu=
        \begin{pNiceArray}{cccccccccc}[first-row,first-col]
                & \scriptstyle 0 & \scriptstyle 1 & \scriptstyle 2 & \scriptstyle 3
                & \scriptstyle 4 & \scriptstyle 5 & \scriptstyle 6
                & \scriptstyle 7 & \scriptstyle 8 & \scriptstyle 9 \\
        \scriptstyle e_1
                & 1&1&0&0&1&0&0&0&0&0\\
        \scriptstyle e_2
                & 0&0&0&0&0&1&1&1&0&0\\
        \scriptstyle e_3
                & 0&0&1&1&0&0&0&0&1&1
        \end{pNiceArray}.
$$

\begin{lemma}\label{lem:q10-mu-characteristic}
    The map $\mu$ is a mod~$2$ characteristic map over $L$.
    The associated small cover $N_\mu=M(Q,\mu)=M^\R(L,\mu)$ is orientable.
\end{lemma}
\begin{proof}
    From the facet list of $L$, each facet of $L$ contains one vertex from each
    of the three sets $\{0,1,4\}$, $\{5,6,7\}$, and $\{2,3,8,9\}$.
    Hence the corresponding three columns of $\mu$ form a basis of $\Z_2^3$, so $\mu$ is a mod~$2$ characteristic map over $L$.

    Also, the sum of the three rows of $\mu$ is $\one_{10}$. 
    Thus $\one_{10}\in\row(\mu)$, and $N_\mu$ is orientable by Corollary~\ref{cor:NN}.
\end{proof}

We now prove that $N_\mu$ fibers over $S^1$.  
The vertices of the cubewise model of $N_\mu$ are indexed by $g\in\Z_2^3$, and the oriented edge in the $i$-direction at $g$ goes from $g$ to $g+\mu_i$.

Let $\epsilon=(0,1,0,0,0,0,1,0,0,0)\in \Z_2^{10}$.
Define an integral $1$-cochain $c$ by $c(g,i)=(-1)^{\langle \mu_i,g\rangle+\epsilon_i}$.
Here $\langle-,-\rangle$ is the standard pairing on $\Z_2^3$.

\begin{lemma} \label{lem:q10-affine-cocycle}
    The cochain $c$ defines a cubewise affine integral Morse function on the universal cover of $N_\mu$.  
    In particular, it defines a homomorphism 
    $$
        [c] \colon \pi_1(N_\mu)\to \Z.
    $$
\end{lemma}
\begin{proof}
    First, $c(g+\mu_i,i)=-c(g,i)$ because $\langle \mu_i,\mu_i\rangle=1$ for $\mu_i\in\{e_1,e_2,e_3\}$.
    Thus the assignment is compatible with reversing the orientation of an edge.
    
    Next let $\{i,j\}$ be an edge of $L$.  
    Since every edge of $L$ joins vertices of different colors, $\mu_i\neq\mu_j$.  
    As $\mu_i,\mu_j$ are distinct standard basis vectors, one has $\langle \mu_i,\mu_j\rangle=0$.
    Therefore 
    $$
            c(g+\mu_j,i)=c(g,i) \quad \text{ and } \quad c(g+\mu_i,j)=c(g,j).
    $$
    Thus opposite parallel edges of every square have the same slope.  
    This is exactly the cubewise affine condition.  
    The induced affine function on the universal cover has no zero slope along any edge, and hence is a cubewise Morse function.
\end{proof}

For each $g\in\Z_2^3$, set
$$
    P_g=\{i\mid c(g,i)>0\} \quad \text{ and } \quad N_g=\{i\mid c(g,i)<0\}.
$$
The values are as follows, where $g = (g_1, g_2, g_3)^T \in \Z_2^3$ is encoded by the integer $g_1+2g_2+4g_3$.
$$
    \begin{array}{c|l|l}
         g & P_g & N_g\\
        \hline
        0 & 0,2,3,4,5,7,8,9 & 1,6\\
        1 & 1,2,3,5,7,8,9 & 0,4,6\\
        2 & 0,2,3,4,6,8,9 & 1,5,7\\
        3 & 1,2,3,6,8,9 & 0,4,5,7\\
        4 & 0,4,5,7 & 1,2,3,6,8,9\\
        5 & 1,5,7 & 0,2,3,4,6,8,9\\
        6 & 0,4,6 & 1,2,3,5,7,8,9\\
        7 & 1,6 & 0,2,3,4,5,7,8,9.
    \end{array}
$$

\begin{lemma}\label{lem:q10-connected-links}
    For every $g\in\Z_2^3$, the induced subgraphs $L_{P_g}^{(1)}$ and $L_{N_g}^{(1)}$ are nonempty and connected.
\end{lemma}
\begin{proof}
    This is checked directly from the edge list, cf. Figure~\ref{fig:Q_L}.  
    The smallest cases are $N_0=\{1,6\}$ and $P_7=\{1,6\}$, and $(1,6)$ is an edge of $L$.  
    Similarly, $N_2=\{1,5,7\}$ and $P_5=\{1,5,7\}$ are connected because $(1,5)$ and $(1,7)$ are edges, and $N_1=\{0,4,6\}$ and $P_6=\{0,4,6\}$ are connected because $(0,6)$ and $(4,6)$ are edges.  
    The remaining induced subgraphs contain one of these connected sets together with vertices adjacent to it, or are checked directly from the same edge list.
\end{proof}

\begin{proposition}\label{prop:q10-N-mu-fibered}
    The small cover $N_\mu$ fibers over $S^1$.
\end{proposition}
\begin{proof}
    By Lemma~\ref{lem:q10-mu-characteristic}, $N_\mu$ is orientable.
    Lemma~\ref{lem:q10-affine-cocycle} shows that $c$ is a nonzero cubewise affine integral $1$-cocycle.
    Also, Lemma~\ref{lem:q10-connected-links} shows that all ascending and descending links are nonempty and connected.

    It remains only to compute the image of $[c]\colon \pi_1(N_\mu)\to \Z$.
    We claim that $\im[c]=2\Z$.
    Modulo $2$, the cochain $c$ takes value $1$ on every oriented edge.
    Since every column of $\mu$ is a standard basis vector, this mod~$2$ cochain is the coboundary of $g\longmapsto g_1+g_2+g_3$.
    Hence $[c]$ takes an even value on every cycle.

    On the other hand, since $\mu_0=\mu_1=e_1$, the difference of the two edges in directions $0$ and $1$ at the vertex $0$ is a cycle $\gamma$.
    By the definition of $c$,
    $$
        \langle c,\gamma\rangle=c(0,0)-c(0,1)=2.
    $$
    Therefore $\im[c]=2\Z$.

    Now Proposition~\ref{prop:fibering-criterion-small-cover}, applied with $d=2$, shows that the primitive class
    $$
        \phi=\frac{1}{2}[c]\in H^1(N_\mu;\Z)
    $$
    is represented by a fibration $N_\mu\to S^1$.
\end{proof}

Let $I$ be an interval with facets $w_0,w_1$.  
Define a characteristic function $\lambda\colon\cF(I\times Q)\to \Z_2^4$ by
$$
    \lambda(w_0)=\lambda(w_1)=e_0 \quad \text{ and } \quad \lambda(u_i)=(0,\mu_i) \quad (i=0,\ldots,9),
$$
where $e_0=(1,0,0,0)$ and $(0,\mu_i)\in \Z_2\oplus \Z_2^3$.  
This is the product characteristic function associated with the interval characteristic function $\varepsilon(w_0)=\varepsilon(w_1)=1$ and the characteristic function $\mu$ on $Q$.

\begin{theorem}\label{thm:q10-product-compatible-example}
    The small cover $M(I\times Q,\lambda)$ is symplectic. 
    Moreover, $I\times Q$ is not combinatorially equivalent to a product of two polygons.
\end{theorem}
\begin{proof}
    Since $\lambda$ is the product characteristic function, we have 
    $$
        M(I\times Q,\lambda)\cong M(I,\varepsilon)\times M(Q,\mu)\cong S^1\times N_\mu.
    $$
    By Proposition~\ref{prop:q10-N-mu-fibered}, the $3$-manifold $N_\mu$ fibers over $S^1$.  
    Thurston's construction for surface bundles over the circle then implies that $S^1\times N_\mu$ admits a symplectic form.  
    Hence $M(I\times Q,\lambda)$ is symplectic.
\end{proof}

\appendix
\section{Computer-assisted verification up to ten facets}
\label{sec-four-dimensional-up-to-ten}

We include a computer-assisted verification for four-dimensional symplectic small covers whose orbit polytopes have at most ten facets.
The example constructed in Theorem~\ref{thm:q10-product-compatible-example} is not covered by this section, since $I \times Q$ has $10+2=12$ facets.

We now discuss the obstructions to symplecticity used in the complete computer-assisted enumeration.
Labels for a simplicial sphere such as \texttt{lutz\_m9\_0057} are used only as identifiers for comparison with Lutz's lists of triangulated manifolds~\cite{LutzManifoldPage3Manifolds,Lutz1999book}.
Each example is given self-containedly by listing the facets of the dual simplicial~complex and the columns of the characteristic matrix.
For brevity, we write $0123$ for the facet $\{0,1,2,3\}$.

Throughout this section, a column of a characteristic matrix over $\Z_2$ may be written as an integer using its binary expansion.
We encode a column vector $(\xi_1,\xi_2,\xi_3,\xi_4)^T \in \Z_2^4$ by the integer $\xi_1 + 2 \xi_2 + 4 \xi_3 + 8 \xi_4$.
For example, the columns in Example~\ref{ex-csymp-not-factor-compatible-P5P4} are
$$
    [1,2,1,2,7,4,8,4,8].
$$

For simplicial spheres with a few vertices, only a few simplices are flag.
We first apply the immediate consequences of Corollary~\ref{cor:only_flag_support_symplectic_small}.
For $m\leq 8$, exactly one is flag.
It is $\partial (I^4)^\ast$.
Hence every small cover over this support is a real Bott manifold.
In particular, every c-symplectic example in this case is factor-compatible.
By Theorem~\ref{thm-two-polygons-factor-compatible}, these examples are symplectic.
Thus no additional obstruction is needed in the $m=8$ case.

For $m=9$, $\partial (P_4 \times P_5)^\ast$ is the only flag simplicial $3$-sphere with nine vertices.
By Theorem~\ref{thm-two-polygons-factor-compatible} and Lemma~\ref{lem:odd-factor-not-compatible}, they do not support symplectic small covers.

For $m=10$, among $247,882$ simplicial $3$-spheres, there are only three flag simplicial $3$-spheres: \texttt{lutz\_m10\_247860}, \texttt{lutz\_m10\_247882} and \texttt{lutz\_m10\_247880}.
The $3$-sphere \texttt{lutz\_m10\_247860} is $\partial(P_4 \times P_6)^\ast$.
Hence, every symplectic small cover over it is factor-compatible by Theorem~\ref{thm-two-polygons-factor-compatible}.
The $3$-sphere \texttt{lutz\_m10\_247882} is indecomposable $3$-sphere with facets
$$
    \begin{gathered}
        0123, 0124, 0135, 0146, 0156, 0237, 0247, 0357, 0467, 0567, 1238, 1248, 1358,\\
        1468, 1568, 2379, 2389, 2479, 2489, 3579, 3589, 4679, 4689, 5679, 5689.
    \end{gathered}
$$
The sphere \texttt{lutz\_m10\_247882} has $25$ facets. 
Hence it cannot support a symplectic small cover by Proposition~\ref{prop-wu-obstruction}.

The remaining $3$-sphere $K$=\texttt{lutz\_m10\_247880} is $\partial (I \times Q')$ where $Q'$ is a simple polytope obtained from a pentagonal prism by an edge cut.
Its facets are
$$
    \begin{gathered}
        0123, 0124, 0135, 0146, 0156, 0237, 0247, 0357, 0468, 0478, 0568, 0578,\\
        1239, 1249, 1359, 1469, 1569, 2379, 2479, 3579, 4689, 4789, 5689, 5789.
    \end{gathered}
$$

\begin{proposition}\label{prop-adjunction-virtual-betti}
    Let $M = \RZ_K/G$ be an oriented $4$-dimensional small cover.
    Assume that $H_2(M;\Q)$ is spanned by the homology classes of embedded tori of self-intersection zero.    
    If $b_1(\RZ_K)>4$, then $M$ is not symplectic.
\end{proposition}
\begin{proof}
    Suppose that $M$ is symplectic.
    By Lemma~\ref{lem:rational-ruled-exclusions}, $M$ is neither rational nor ruled.
    The adjunction inequality applied to the spanning square-zero tori would then force the symplectic canonical class to vanish in $H^2(M;\Q)$.
    The pullback of the canonical class vanishes in $H^2(\RZ_K;\Q)$.
    Hence it is torsion in integral cohomology.
    Thus $\RZ_K$ would be a symplectic Calabi--Yau $4$-manifold in the sense that its canonical class is torsion.
    By the Bauer--Li restriction, as formulated for example in~\cite{Friedl-Vidussi2013}, a symplectic Calabi--Yau $4$-manifold with positive first Betti number satisfies
    $$
        b_1(\RZ_K)\le 4.
    $$
    This contradicts the hypothesis.
\end{proof}
\begin{example}\label{ex-virtual-betti}
    Let $K$ be the simplicial~$3$-sphere labelled \texttt{lutz\_m10\_247880}.
    Take
    $$
        \Lambda=[1,2,4,8,14,14,4,2,8,1].
    $$
    It is straightforward to check that $\Lambda$ satisfies the characteristic condition.
    The corresponding small cover is orientable and has $b_2(M)=2$. 
    Hence it is c-symplectic by Proposition~\ref{prop-four-manifold-csymp}.
    On the other hand, $b_1(\RZ_K)=32>4$.
    
    A direct cellular computation shows that $H_2(M;\Q)$ is spanned by embedded square-zero tori.
    Therefore Proposition~\ref{prop-adjunction-virtual-betti} applies, and $M$ admits no symplectic structure.
\end{example}

By a direct computation, $K$ supports exactly $100$ c-symplectic characteristic matrices up to D--J~equivalence.
For each of them, the homology group $H_2(M;\Q)$ is spanned by embedded square-zero tori.
Since $b_1(\RZ_K)=32$, Proposition~\ref{prop-adjunction-virtual-betti} eliminates all these cases.
Example~\ref{ex-virtual-betti} records one representative case.

Therefore, the following theorem immediately follows.
\begin{theorem}\label{thm-up-to-ten-factor-compatible}
    Let $M=M(P,\lambda)$ be a four-dimensional small cover over a simple $4$-polytope $P$ with at most ten facets.
    If $M$ is symplectic, then $P$ is a product of two polygons and $M$ is factor-compatible.
\end{theorem}

\begin{remark}
    For $m=11$, the Sulanke--Lutz~\cite{Sulanke-Lutz2009} enumeration contains $166,564,303$ simplicial~$3$-spheres.
    Although the list is available at~\cite{LutzManifoldPage3Manifolds}, the number of candidates is too large for the certificate-by-certificate verification carried out in this section.
    We therefore stop the complete enumeration at $m=10$.
\end{remark}

\providecommand{\bysame}{\leavevmode\hbox to3em{\hrulefill}\thinspace}
\providecommand{\MR}{\relax\ifhmode\unskip\space\fi MR }
\providecommand{\MRhref}[2]{%
  \href{http://www.ams.org/mathscinet-getitem?mr=#1}{#2}
}
\providecommand{\href}[2]{#2}

\end{document}